\documentclass[reqno]{amsart}
\usepackage{amssymb}
\usepackage{graphicx}

\usepackage[usenames, dvipsnames]{color}
\usepackage{verbatim}
\usepackage{mathrsfs}

\numberwithin{equation}{section}

\newtheorem{theorem}{Theorem}[section]

\newtheorem{lemma}[theorem]{Lemma}
\newtheorem{prop}[theorem]{Proposition}

\theoremstyle{definition}
\newtheorem{remark}[theorem]{Remark}

\theoremstyle{definition}

\theoremstyle{definition}

\makeatletter
\def\dashint{\operatorname%
{\,\,\text{\bf-}\kern-.98em\DOTSI\intop\ilimits@\!\!}}
\makeatother

\def\\det{\text{det}}

\def\.5{\frac{1}{2}}

\def\bR{\mathbb{R}}

\def\cB{\mathcal{B}}

\def\cD{\mathcal{D}}

\def\cG{\mathcal{G}}

\newcommand{\RN}[1]{%
  \textup{\uppercase\expandafter{\romannumeral#1}}%
}

\renewcommand{\epsilon}{\varepsilon}

\newcounter{marnote}

\begin{document}
\title[Estimates by Green function method]{Optimal estimates for the conductivity problem by Green's function method}
\author[H. Dong]{Hongjie Dong}
\address[H. Dong]{Division of Applied Mathematics, Brown University,
182 George Street, Providence, RI 02912, USA. Tel: 1-(401)8637297}
\email{Hongjie\_Dong@brown.edu}

\author[H.G. Li]{Haigang Li}
\address[H.G. Li]{School of Mathematical Sciences, Beijing Normal University, Laboratory of Mathematics and Complex Systems, Ministry of Education, Beijing 100875, China.}
\email{hgli@bnu.edu.cn}


\date{\today} 

\begin{abstract}
We study a class of second-order elliptic equations of divergence form, with discontinuous coefficients and data, which models the conductivity problem in composite materials. We establish optimal gradient estimates by showing the explicit dependence of the elliptic coefficients and the distance between interfacial boundaries of inclusions. The novelty of these estimates is that they unify the known results in the literature and answer open problem $(b\,)$ proposed by Li-Vogelius (2000) for the isotropic conductivity problem. We also obtain more interesting higher-order derivative estimates, which answers open problem $(c\,)$ of Li-Vogelius (2000). It is worth pointing out that the equations under consideration in this paper are nonhomogeneous.
\end{abstract}

\maketitle

\section{Introduction and main results}

In this paper, we establish optimal gradient and higher derivative estimates for solutions of the isotropic conductivity problem. The problem is modeled by a class of divergence form second-order elliptic equations with discontinuous coefficients and data
\begin{equation}\label{mainequ}
L_{\varepsilon;r_{1},r_{2}}u:=D_{i}(a(x)D_{i}u)=D_{i}f_{i}\qquad\mbox{in}~~\mathcal{D},
\end{equation}
where the Einstein summation convention in repeated indices is used, $\mathcal{D}$ is a bounded open subset of $\mathbb{R}^{2}$,
$$
a(x)=k_{1}\chi_{\mathcal{B}_{1}}+k_{2}\chi_{\mathcal{B}_{2}}
+\chi_{\cB_0},
$$
$k_{1},k_{2},r_1,r_2\in (0,\infty)$ are constants,
$$
\mathcal{B}_{1}:=B_{r_{1}}({\varepsilon/2}+r_{1},0),\quad
\mathcal{B}_{2}:=B_{r_{2}}(-{\varepsilon/2}-r_{2},0),\quad
\cB_0=\mathbb{R}^{2}\setminus(\overline{\mathcal{B}_{1}\cup\mathcal{B}_{2}}),
$$
and $\chi$ is the indicator function. Here $\mathcal{D}$ models the cross-section of a fiber-reinforced composite, with the disks $\mathcal{B}_{1}$ and $\mathcal{B}_{2}$ representing the cross-sections of the fibers; the remaining subdomain representing the matrix surrounding the fibers. The gradient of the potential $u$ represents the electric field in the conductivity problem or the stress in anti-plane elasticity problem. Moreover, $a(x)$ is the conductivity (for the conductivity problem) or the shear modulus (for the anti-plane shear problem), which is a constant on the fibers, and a different constant on the matrix. The constant $\varepsilon$ is used to denote the distance between $\mathcal{B}_{1}$ and $\mathcal{B}_{2}$. See Figure \ref{fig:1.1}. It is important from a practical point of view to know whether  $|Du|$ can be arbitrarily large as the inclusions get closer to each other. It is also of interest to establish similar estimates for higher-order norms of solutions. The purpose of this paper is to investigate the explicit dependence of $|D^{m}u|$ $(m\geq1)$ on $\epsilon,k_{1},k_{2}$, $r_{1}$, and $r_{2}$.

\begin{figure}
\begin{center}
\includegraphics[width=2.5in]{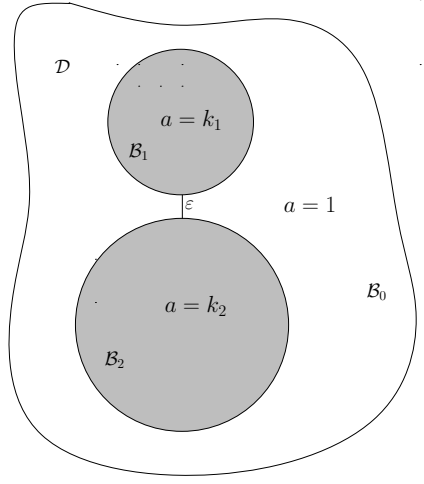}
\caption{A domain with two closely spaced inclusions}
\label{fig:1.1}
\end{center}
\end{figure}

In \cite{ba}, Babu\u{s}ka, Andersson, Smith, and Levin numerically analyzed the initiation and growth of damage in composite materials, in which the inclusions are frequently spaced very closely and even touching. There have been many important work on the gradient estimates for solutions of elliptic equations and systems arising from composite materials. See, for instance,  \cite{bll,bv,d,llby,ln,lv} and the references therein.

For two touching disks in 2D with $k_{1}=k_{2}=k$ away from $0$ and $\infty$, Bonnetier and Vogelius \cite{bv} proved that $|Du|$ remains bounded, with an upper bound depending on the value of $k$. Li and Vogelius \cite{lv} extended the result to general divergence form second-order elliptic equations with piecewise H\"older continuous coefficients in all dimensions, and they proved that $|Du|$ remains bounded as $\varepsilon\searrow 0$. Actually they established stronger, $\varepsilon$-independent, piecewise $C^{1,\alpha}$ estimates for solutions. This extension covers inclusions of arbitrary smooth shape. Later Li-Nirenberg \cite{ln} further extended these results to general divergence form second-order elliptic systems including systems of elasticity.
The estimates in \cite{ln} and \cite{lv} all depend on the ellipticity of the coefficients. In \cite[Page 94]{lv}, Li and Vogelius proposed several interesting questions including the following two:

\noindent (b\,): How does the constant in the estimates depend on the ellipticity constants?

\noindent (c\,): Do similar estimates
hold for higher order norms of the solution, assuming of course all the data are appropriately smooth?

On the other hand, if the ellipticity constants are allowed to partially deteriorate, the situation is very different.
The perfect conductivity problem with two inclusions can be described as follows
\begin{equation}
                    \label{eq9.46}
\begin{cases}
\Delta u=0 \quad &\text{in}~\,\cD\setminus (\overline{\cB_1\cup \cB_2}),\\
Du=0 \quad&\text{in}\,\,\cB_1\cup \cB_2,\\
u|_+=u|_-\quad &\text{on}\,\,\partial \cB_1\cup \partial \cB_2,\\
\int_{\partial \cB_i}(\partial u/\partial \nu)|_+=0\quad&\text{for}\,\,i=1,2.
\end{cases}
\end{equation}
Formally, the system \eqref{eq9.46} above can be obtained from \eqref{mainequ} by setting $f_i=0$ and passing to the limit as $k_1,k_2\nearrow \infty$.
In contrast to the case when $k_1$ and $k_2$ are finite and bounded below from zero,
it was shown in \cite{BC84} and \cite{Ma96} that the gradient of solutions may blow up as two inclusions approach each other with the blow-up rate $\varepsilon^{-1/2}$ for a special solution in the 2D case. Rigorous proofs were later carried out by Ammari, Kang, and Lim \cite{akl}, and Ammari, Kang, Lim, Lee, and Lee \cite{aklll} for the case of circular inclusions. Since then, the problem has been studied by many mathematician. It has been proved that for the two close-to-touching inclusions case the generic rate of $|Du|$ blow-up is $\varepsilon^{-1/2}$ in two dimensions, $|\varepsilon\log\varepsilon|^{-1}$ in three dimensions, and $\varepsilon^{-1}$ in dimensions greater than three. See Yun \cite{y1,y2}, Bao, Li, and Yin \cite{bly1}, as well as Lim and Yun \cite{ly}, and Ammari, Bonnetier, Triki, and Vogelius \cite{abtv}. We also mention that more detailed characterizations of the singular behavior of $|Du|$ have been obtained by Ammari, Ciraolo, Kang, Lee, and Yun \cite{ackly}, Ammari, Kang, Lee, Lim, and Zribi \cite{akllz}, Bonnetier and Triki \cite{bt0,bt} and Kang, Lim, and Yun \cite{kly,kly2}.

Similar to the perfect conductivity problem, the insulated conductivity problem can also be derived from \eqref{mainequ} by passing to the limit as $k_1,k_2\searrow 0$, which can be considered as the complementary problem to the perfect case. The corresponding system is then given by
\begin{equation}
                    \label{eq10.18}
\begin{cases}
\Delta u=0 \quad &\text{in}~\,\cD\setminus (\overline{\cB_1\cup \cB_2}),\\
\partial u/\partial\nu=0\quad  &\text{on}\,\,\partial \cB_1\cup \partial \cB_2.
\end{cases}
\end{equation}
As in the perfect case, the gradient of solutions to \eqref{eq10.18} generally blows up as the distance between the inclusions goes to zero. In 2D, the authors of \cite{akl,aklll} obtained the optimal bound for circular inclusions with comparable radii and the blow-up rate is $\varepsilon^{-1/2}$. The proof uses harmonic conjugators to convert the insulated case to the perfect case, which no longer works when $d\ge 3$. For the higher dimensional case, Bao, Li, and Yin \cite{bly2} established upper bound $\varepsilon^{-1/2}$ by using a flipping technique and a variant of Li-Nirenberg's result. We point out that in all these papers, the equation is assumed to be homogeneous, i.e., $f_i\equiv 0$, $i=1,2$.

The main purpose of this paper is to give a uniform expression for these three situations. The key strategy is to construct a Green's function for the elliptic operator $L_{\varepsilon;r_{1},r_{2}}$ with piecewise constant coefficients in order to express the solution explicitly. We then establish unified gradient estimates for the solutions of \eqref{mainequ}, which involve a precise dependence on $k_{1},k_{2},r_{1},r_{2}$, and $\varepsilon$. This answers the open problem $(b\,)$ proposed by Li-Vogelius \cite{lv} in the case of circular inclusions. Furthermore, regarding the open problem $(c\,)$, we obtain an upper bound for higher derivatives, which extends the results of Li-Vogelius \cite{lv} and Dong-Zhang \cite{dz}, where the case when $\varepsilon=0$ was studied.

Let $k_0=1$,
$$\alpha=\frac{k_1-1}{k_1+1},\quad\mbox {and} \quad\beta=\frac{k_2-1}{k_2+1}.$$
To illustrate the main ideas of the proof, we first assume that $r_{1}=r_{2}=1$ and show the dependence of $|Du|$ and $|D^{m}u|$ with respect to $k_{1}$, $k_{2}$, and $\varepsilon$. The general case is treated in Theorem \ref{mainthm3}.

In the first theorem below, we consider a weak solution to \eqref{mainequ} in a neighborhood of the origin, which may not entirely contain the balls $\cB_1$ and $\cB_2$.

\begin{theorem}\label{mainthm1}
Let $\varepsilon\in (0,1/2)$ and $\gamma\in (0,1)$ be constants. Assume that $u$ is a weak solution of \eqref{mainequ} in $B_1:=B_{1}(0)$ with $r_{1}=r_{2}=1$ and $f$ is piecewise $C^\gamma$ in $B_1$, which satisfy for some constant $C_1>0$,
\begin{equation}
                            \label{eq11.56}
\|u\|_{L_2(B_1)}\le C_1,\quad \|f\|_{C^\gamma(B_1\cap \cB_j)}\le C_{1}k_j,\quad j=0,1,2.
\end{equation}
Then we have
\begin{equation}\label{equ_mainresult}
|Du|\le \frac{CC_1}{1-(1-\sqrt\varepsilon)|\alpha\beta|}\quad\text{in}\,\,B_{1/2},
\end{equation}
where $C>0$ is a constant depending only on $\gamma$.
Furthermore, for any integer $m\geq2$, if $f$ is piecewise $C^{m-1,\gamma}$ in $B_1$, and for some constant $C_m>0$,
\begin{equation}
                                    \label{eq4.43}
\|u\|_{L_2(B_1)}\le C_m,\quad \|f\|_{C^{m-1,\gamma}(B_1\cap \cB_j)}\le C_m\,k_j,\quad j=0,1,2,
\end{equation}
then we have
\begin{equation}\label{equ_mainresultm}
|D^{m}u|\le
\frac{CC_m}{\big(1-(1-\sqrt\varepsilon)|\alpha\beta|\big)^m}
\quad\text{in}\,\,B_{1/2},
\end{equation}
where $C>0$ is a constant depending only on $m$ and $\gamma$.
\end{theorem}

In the next theorem, we obtain more precise estimates in $\cB_1$ and $\cB_2$ by assuming that $u$ satisfies \eqref{mainequ} in a domain which contains both $\cB_1$ and $\cB_2$.

\begin{theorem}\label{mainthm2}
Let $\varepsilon\in (0,1/2)$ and $\gamma\in (0,1)$ be constants. Assume that $\cB_1\cup \cB_2\Subset \cD_1\Subset \cD_2\Subset \cD$ for some domains $\cD_1$ and $\cD_2$, $u$ is a weak solution of \eqref{mainequ} in $\cD$ with $r_{1}=r_{2}=1$, and $f$ is piecewise $C^\gamma$ in $\cD$, which satisfy for some constant $C_1>0$,
\begin{equation}\label{def_C1}
\|u\|_{L_2(\cD)}\le C_1,\quad \|f\|_{C^\gamma(\cD\cap \cB_j)}\le C_1\min\{1,k_j\},\quad j=0,1,2.
\end{equation}
Then we have
\begin{equation}\label{equ_mainresultb}
|Du|\le
\begin{cases}
\frac{CC_1}{1-(1-\sqrt\varepsilon)|\alpha\beta|}
&\qquad\mbox{in}~\cD_{1}\cap \cB_0,\\
\frac{1}{k_{1}+1}\cdot\frac{CC_1}{1-(1-\sqrt\varepsilon)|\alpha\beta|}
&\qquad\mbox{in}~\mathcal{B}_{1},\\
\frac{1}{k_{2}+1}\cdot\frac{CC_1}{1-(1-\sqrt\varepsilon)|\alpha\beta|}
&\qquad\mbox{in}~\mathcal{B}_{2},
\end{cases}
\end{equation}
where $C>0$ is a constant depending only on $\gamma$, $\cD_1$, $\cD_2$, and $\cD$.
Furthermore, for any integer $m\geq2$, if $f$ is piecewise $C^{m-1,\gamma}$ in $\cD$, and for some constant $C_m>0$,
\begin{equation}\label{def_Cm}
\|u\|_{L_2(\cD)}\le C_m,\quad \|f\|_{C^{m-1,\gamma}(\cD\cap \cB_j)}\le C_m\min\{1,k_j\},\quad j=0,1,2,
\end{equation}
then we have
\begin{equation}\label{equ_mainresultmb}
|D^{m}u|\le
\begin{cases}
\frac{CC_m}{\left(1-(1-\sqrt\varepsilon)|\alpha\beta|\right)^m}
&\qquad\mbox{in}~ \cD_{1}\cap\cB_0,\\
\frac{1}{k_{1}+1}\cdot\frac{CC_m}{\left(
1-(1-\sqrt\varepsilon)|\alpha\beta|\right)^m}
&\qquad\mbox{in}~\mathcal{B}_{1},\\
\frac{1}{k_{2}+1}\cdot\frac{CC_m}{\left(
1-(1-\sqrt\varepsilon)|\alpha\beta|\right)^m}&\qquad\mbox{in}~\mathcal{B}_{2},
\end{cases}
\end{equation}
where $C>0$ is a constant depending only on $m$, $\gamma$, $\cD_1$, $\cD_2$, and $\cD$.
\end{theorem}

In the case of circular inclusions, the estimates \eqref{equ_mainresult} and \eqref{equ_mainresultb} unify the known results in the literature:
\begin{enumerate}
\item boundedness of $|Du|$ for finite $k_{1}$ and $k_{2}$ regardless of the distance $\varepsilon$ obtained in \cite{bv,ln,lv};

\item  blow-up of $|Du|$ with rate of $1/\sqrt{\varepsilon}$ for $k_{1}=k_{2}=+\infty$ established in \cite{ackly,akl,aklll,bly1,y1,y2};

\item blow-up of $|Du|$ with rate of $1/\sqrt{\varepsilon}$ for $k_{1}=k_{2}=0$ established \cite{akl,aklll}.
\end{enumerate}
This is the first main contribution of this paper. It is worth pointing out that the estimate in $\mathcal{D}\cap\mathcal{B}_{0}$ of \eqref{equ_mainresultb} was achieved by using a completely different method, single layer potential method, in \cite{akl,aklll,bly1,y1,y2}. Compared to the references mentioned above, we also consider more general non-homogeneous equations.
The second main contribution of our paper is the higher-order derivative estimates \eqref{equ_mainresultm} and \eqref{equ_mainresultmb}, extending the results of Li-Vogelius \cite{lv} and more recent work of Dong-Zhang \cite{dz}, where the boundedness of the higher derivatives was proved when two balls are assumed to touch each other at the origin. This in particular answers in the affirmative the conjecture in Li-Vogelius \cite[Remark 8.2, pp 137]{lv}:
\begin{quote}
{\em ``We do feel, however, that for $0 < a_0 < 1$ (as is the case here) the smoothness exhibited by $u_0$ makes it quite likely that the $u_\varepsilon$ have piecewise defined, uniformly bounded derivatives of any order...''}
\end{quote}
Here {\em $u_{0}$} stands for the solution when $\varepsilon=0$, i.e., the two balls touch each other.

\begin{remark}
The following example gives a lower bound of the gradient, which shows that \eqref{equ_mainresult} is optimal when $k_1,k_2\ge 1$. Choose $f_1$ to be a nonnegative function supported in $B_{1/10}(-3,0)$ which is a even function in $x_2$ with unit integral, and $f_2\equiv 0$.
Assume that $u$ is a weak solution to \eqref{mainequ} with $f$ defined above and $\mathcal{B}_{1}\cup\mathcal{B}_{2}\subset\mathcal{D}$. Then we have
\begin{equation}\label{equ_mainresult_lower}
|Du(0)|\geq
\frac{C}{1-(1-\sqrt\varepsilon)\alpha\beta}
\end{equation}
for some constant $C>0$ independent of $\varepsilon$, $k_1$, and $k_2$.
See Section \ref{subsec_lowerbound} for more details. For the special case when $k_{1}=k_{2}=k$, we can rewrite \eqref{equ_mainresult} and \eqref{equ_mainresult_lower} as follows to get more transparent dependence of $|Du|$ on $k$ and $\varepsilon$: when $k\gg 1$,
\begin{equation*}
|Du|\le \frac{CC_1}{\sqrt\varepsilon+1/k}\quad\text{in}\,\,B_{1/2}
\quad\mbox{and}\quad
|Du(0)|\geq
\frac{C}{\sqrt\varepsilon+1/k};
\end{equation*}
when $k\ll 1$,
\begin{equation*}
|Du|\le \frac{CC_1}{\sqrt\varepsilon+k}\quad\text{in}\,\,B_{1/2}
\quad\mbox{and}\quad
|Du(0)|\geq
\frac{C}{\sqrt\varepsilon+k}.
\end{equation*}
\end{remark}

\begin{remark}
When $\mathcal{B}_{1}\cup\mathcal{B}_{2}\subset\mathcal{D}$ and $\varepsilon$ is fixed small, it is easy to see from \eqref{equ_mainresultb} that as $k_{1},k_{2}\to \infty$, $\alpha,\beta\to 1$. Therefore, we not only have the blow-up rate $1/\sqrt{\varepsilon}$ in the narrow region $\mathcal{D}\cap\mathcal{B}_{0}$, but also show that the upper bounds in $\mathcal{B}_{1}$ and $\mathcal{B}_{2}$ tend to zero as $k_1,k_2\to \infty$, which are consistent with the condition in \cite{bly1} that $Du=0$ in  $\mathcal{B}_{1}\cup\mathcal{B}_{2}$, derived by a variational argument. See \cite[the Appendix]{bly1}.
\end{remark}

Finally we give the following theorem for the general case that $r_{1}$ and $r_{2}$ are not necessarily equal to $1$.

\begin{theorem}\label{mainthm3}

Let $\varepsilon\in (0,1/2)$, $\gamma\in (0,1)$ be constants and $\varepsilon\ll\,r_{1},r_{2}<10$. Under the assumptions of Theorem \ref{mainthm2}, we have, for any integer $m\geq1$,
\begin{equation*}
|D^{m}u|\le
\begin{cases}
\frac{CC_m}{\left(1-\big(1-\sqrt{2(1/r_{1}
+1/r_{2})\varepsilon}\big)|\alpha\beta|\right)^m}
&\qquad\mbox{in}~ \cD_{1}\cap\cB_0,\\
\frac{CC_m}{k_{1}+1}\cdot\frac{1}{\left(
1-\big(1-\sqrt{2(1/r_{1}
+1/r_{2})\varepsilon}\big)|\alpha\beta|\right)^m}
&\qquad\mbox{in}~\mathcal{B}_{1},\\
\frac{CC_m}{k_{2}+1}\cdot\frac{1}{\left(
1-\big(1-\sqrt{2(1/r_{1}
+1/r_{2})\varepsilon}\big)|\alpha\beta|\right)^m}&\qquad\mbox{in}~\mathcal{B}_{2},
\end{cases}
\end{equation*}
where $C>0$ is a constant depending only on $m$, $\gamma$, $\cD_1$, $\cD_2$, and $\cD$.
\end{theorem}

The remaining part of the paper is organized as follows. Section \ref{sec_Green} is devoted to the construction of a Green's function of the operator $L_{\varepsilon;r_{1},r_{2}}$ by using the inversion maps $\Phi_1$ and $\Phi_2$ with respect to the circles $\partial\cB_1$ and $\partial\cB_2$. In Section \ref{sec_thm12} we first derive the derivative estimates of the maps $\Phi_{1}\Phi_{2}$ and $\Phi_{2}\Phi_{1}$, as well as their compositions. We then prove our main results, Theorems \ref{mainthm1} and \ref{mainthm2}. Finally, Theorem \ref{mainthm3} is proved in Section \ref{sec_thm3}.

\section{Construction of a Green's function}\label{sec_Green}

In this section, we construct a Green's function of the operator $L_{\varepsilon;r_{1},r_{2}}$ by adapting an idea in \cite{lv}. It is well known that
$$-\frac{1}{2\pi}\Delta\log|x-y|=\delta(x-y).$$
For simplicity of exposition, we write $\Delta\log|x-y|=\delta(x-y)$.
In order to illustrate the main ideas, we assume that $r_{1}=r_{2}=1$. The general case is similar. See Section \ref{sec_thm3}.  We use $L$ to denote the operator $L_{\varepsilon;1,1}$.

Define the inversion maps of a point $x\in\mathbb{R}^{2}$ with respect to $\partial\mathcal{B}_{1}=\partial{B}_{1}(1+{\varepsilon/2},0)$ and $\partial\mathcal{B}_{2}=\partial{B}_{1}(-1-{\varepsilon/2},0)$, respectively, by
\begin{equation}\label{Phi1real}
\Phi_{1}(x_{1},x_{2}):=\left(\frac{x_{1}-(1+ \varepsilon/2)}{(x_{1}-1- \varepsilon/2)^{2}+x_{2}^{2}}+1+\varepsilon/2,~\frac{x_{2}}{(x_{1}-1- \varepsilon/2)^{2}+x_{2}^{2}}\right)
\end{equation}
and
$$\Phi_{2}(x_{1},x_{2}):=\left(~\frac{x_{1}+1+ \varepsilon/2}{(x_{1}+1+ \varepsilon/2)^{2}+x_{2}^{2}}-1-\varepsilon/2,~\frac{x_{2}}{(x_{1}+1+ \varepsilon/2)^{2}+x_{2}^{2}},\right).$$

Recall that
$$
\alpha=\frac{k_{1}-1}{k_{1}+1}\quad\text{and}\quad\beta=\frac{k_{2}-1}{k_{2}+1}.
$$
It is clear that $\alpha,\beta\in(-1,1)$. We define an auxiliary function $\mathcal{G}(x,y)$ as follows.

(1) When $y\in\mathcal{B}_{0}$, $\mathcal{G}(x,y)$ equals
\begin{align*}
&{\frac 2 {k_{1}+1}}\sum_{l=0}^{\infty}(\alpha\beta)^{l}
\Big(\log|(\Phi_{1}\Phi_{2})^{l}(x)-y|
-\beta\log|(\Phi_{2}\Phi_{1})^{l}\Phi_{2}(x)-y|\Big)\\
&\quad\hspace{8cm}\mbox{for}~x\in\overline{\mathcal{B}}_{1};\\
&\log|x-y|+\sum_{l=1}^{\infty}\Big[(\alpha\beta)^{l}
\Big(\log|(\Phi_{1}\Phi_{2})^{l}(x)-y|+\log|(\Phi_{2}\Phi_{1})^{l}(x)-y|\Big)\\
&\quad -(\alpha\beta)^{l-1}\Big(\beta\log|(\Phi_{2}\Phi_{1})^{l-1}\Phi_{2}(x)-y|
+\alpha\log|(\Phi_{1}\Phi_{2})^{l-1}\Phi_{1}(x)-y|\Big)\Big]\\
&\quad\hspace{8cm}\mbox{for}~x\in\mathcal{B}_{0};\\
&{\frac 2 {k_{2}+1}}\sum_{l=0}^{\infty}(\alpha\beta)^{l}
\Big(\log|(\Phi_{2}\Phi_{1})^{l}(x)-y|-
\alpha\log|(\Phi_{1}\Phi_{2})^{l}\Phi_{1}(x)-y|\Big)\\
&\quad\hspace{8cm}\mbox{for}~x\in\overline{\mathcal{B}}_{2}.
\end{align*}

(2) When $y\in\mathcal{B}_{1}$, $\mathcal{G}(x,y)$ equals
\begin{align*}
&{\frac 1 {k_1}}\big(\log|x-y|+\alpha\log|\Phi_{1}(x)-y|\big)
-{\frac{4\beta}{(k_{1}+1)^2}}\sum_{l=0}^{\infty}(\alpha\beta)^{l}
\log|(\Phi_{2}\Phi_{1})^{l}\Phi_{2}(x)-y|\\
&\quad\hspace{7.3cm}\mbox{for}~x\in\overline{\mathcal{B}}_{1}\setminus\{(1+\varepsilon/2,0)\};\\
&{\frac 2 {k_{1}+1}}\sum_{l=0}^{\infty}(\alpha\beta)^{l}
\Big(\log|(\Phi_{2}\Phi_{1})^{l}(x)-y|
-\beta\log|(\Phi_{2}\Phi_{1})^{l}\Phi_{2}(x)-y|\Big)\\
&\quad\hspace{7.3cm}\mbox{for}~x\in\mathcal{B}_{0};\\
&{\frac{4}{(k_{1}+1)(k_{2}+1)}}\sum_{l=0}^{\infty}(\alpha\beta)^{l}
\log|(\Phi_{2}\Phi_{1})^{l}(x)-y|\qquad\mbox{for}~x\in\overline{\mathcal{B}}_{2}.
\end{align*}

(3) When $y\in\mathcal{B}_{2}$, $\mathcal{G}(x,y)$ equals
\begin{align*}
&{\frac{4}{(k_{1}+1)(k_{2}+1)}}
\sum_{l=0}^{\infty}(\alpha\beta)^{l}
\log|(\Phi_{1}\Phi_{2})^{l}(x)-y|\qquad\mbox{for}~x\in\overline{\mathcal{B}}_{1};\\
&{\frac 2 {k_{2}+1}}\sum_{l=0}^{\infty}(\alpha\beta)^{l}
\Big(\log|(\Phi_{1}\Phi_{2})^{l}(x)-y|
-\alpha\log|(\Phi_{1}\Phi_{2})^{l}\Phi_{1}(x)-y|\Big)\\
&\quad\hspace{7.2cm}\mbox{for}~x\in\mathcal{B}_{0};\\
&{\frac 1 {k_2}}\big(\log|x-y|+\beta\log|\Phi_{2}(x)-y|\big)
-{\frac{4\alpha}{(k_{2}+1)^2}}\sum_{l=0}^{\infty}(\alpha\beta)^{l}
\log|(\Phi_{1}\Phi_{2})^{l}\Phi_{1}(x)-y|\\
&\quad\hspace{7.2cm}\mbox{for}~x\in\overline{\mathcal{B}}_{2}
\setminus\{(1+\varepsilon/2,0)\}.
\end{align*}

\begin{remark}\label{rem1}
We note that in the above definition of $\mathcal{G}$, for the case $y\in\mathcal{B}_{1}$ (or $y\in\mathcal{B}_{2}$), the point $(1+\varepsilon/2,0)$ (or $(0,-1-\frac{\epsilon}{2})$, respectively) is excluded, because $\Phi_{1}$ (or $\Phi_{2}$) has a singularity at $(1+\varepsilon/2,0)$ (or $(0,-1-\frac{\epsilon}{2})$, respectively). All the other terms appearing in the summations are regular.
\end{remark}

\begin{lemma}\label{lem1}
Let $\mathcal{G}(x,y)$ be defined above. Then we have
\begin{align*}
\Delta_{x}\mathcal{G}(x,y)=&\delta(x-y)\quad\mbox{for}~y\in\mathcal{B}_{0}
~\mbox{and}~x\notin\partial(\mathcal{B}_{1}\cup\mathcal{B}_{2});\\
k_{1}\Delta_{x}\mathcal{G}(x,y)=&\delta(x-y)\quad\mbox{for}~y\in\mathcal{B}_{1}
~\mbox{and}~x\notin\partial(\mathcal{B}_{1}\cup\mathcal{B}_{2})\cup
\{(1+\varepsilon/2,0)\};\\
k_{2}\Delta_{x}\mathcal{G}(x,y)=&\delta(x-y)\quad\mbox{for}~y\in\mathcal{B}_{2}
~\mbox{and}~x\notin\partial(\mathcal{B}_{1}\cup\mathcal{B}_{2})\cup
\{(-1-\varepsilon/2,0)\}.
\end{align*}
Moreover, $\mathcal{G}(x,y)$ and $a(x)D_{\nu}\mathcal{G}(x,y)$ are continuous across $\partial\mathcal{B}_{1}\cup\partial\mathcal{B}_{2}$, where $\nu$ is the unit norm vector field of $\partial(\mathcal{B}_{1}\cup\mathcal{B}_{2})$
\end{lemma}

\begin{proof}
We first consider the case when $y\in\mathcal{B}_{0}$. When $x\in\mathcal{B}_{1}$,
$(\Phi_{1}\Phi_{2})^{l}(x)\in\mathcal{B}_{1}$ and $(\Phi_{2}\Phi_{1})^{l}\Phi_{2}(x)\in\mathcal{B}_{2}$.
Therefore, 
$\mathcal{G}(x,y)$ is harmonic in $\mathcal{B}_{1}$. In the same way, we can show that $\mathcal{G}(x,y)$ is harmonic in $\mathcal{B}_{2}$ as well. When $x\in\mathcal{B}_{0}$, each term in the expression of $\mathcal{G}(x,y)$, with the exception of $\log|x-y|$, is harmonic in $\mathcal{B}_{0}$ by the same argument. Hence, when $x\in\mathcal{B}_{0}$, $\Delta_{x}\mathcal{G}(x,y)=\delta(x-y)$. Recall that, as we mention at the beginning of this section, we write $\Delta_{x}\log|x-y|=\delta(x-y)$.

It remains to verify the continuity of $\mathcal{G}(x,y)$ and $a(x)D_{\nu}\mathcal{G}(x,y)$ across the two circles $\partial \cB_1$ and $\partial \cB_2$. We only present the calculations corresponding to the case $\partial\mathcal{B}_{1}$ because the other case is similar. Using the simple fact that
\begin{equation}\label{phi1x=x}
\Phi_1(x)=x\quad\mbox{on}~~\partial \mathcal B_1,
\end{equation}
we have, for any $y\in\mathcal{B}_{0}$,
\begin{align*}
\mathcal{G}(&x,y)\Big|_{\partial\mathcal{B}_{1}}^{-}\\
=&\log|x-y|+\sum_{l=1}^{\infty}\Big[(\alpha\beta)^{l}
\Big(\log|(\Phi_{1}\Phi_{2})^{l}(x)-y|+\log|(\Phi_{2}\Phi_{1})^{l}(x)-y|\Big)\\
&-(\alpha\beta)^{l-1}\Big(\beta\log|(\Phi_{2}\Phi_{1})^{l-1}\Phi_{2}(x)-y|+\alpha\log|(\Phi_{1}\Phi_{2})^{l-1}\Phi_{1}(x)-y|\Big)\Big]\\
=&\log|x-y|+\sum_{l=1}^{\infty}\Big[(\alpha\beta)^{l}
\Big(\log|(\Phi_{1}\Phi_{2})^{l}(x)-y|+\log|(\Phi_{2}\Phi_{1})^{l-1}\Phi_{2}(x)-y|\Big)\\
&-(\alpha\beta)^{l-1}\Big(\beta\log|(\Phi_{2}\Phi_{1})^{l-1}\Phi_{2}(x)-y|
+\alpha\log|(\Phi_{1}\Phi_{2})^{l-1}(x)-y|\Big)\Big]\\
=&(1-\alpha)\sum_{l=0}^{\infty}(\alpha\beta)^{l}
\Big(\log|(\Phi_{1}\Phi_{2})^{l}(x)-y|-\beta\log|(\Phi_{2}\Phi_{1})^{l}\Phi_{2}(x)-y|\Big)\\
=&\mathcal{G}(x,y)\Big|_{\partial\mathcal{B}_{1}}^{+},
\end{align*}
where we used $1-\alpha=2/(k_{1}+1)$ in the last equality.

Next, we check the continuity of $a(x)D_{\nu}\mathcal{G}(x,y)$ on $\partial\mathcal{B}_{1}$. Notice that  for any differentiable function $f$,
\begin{equation}\label{normalB1}
\partial_\nu \big(f(\Phi_1(x))\big)=-\partial_\nu f(x) \quad\mbox{on}~~\partial \mathcal{B}_{1}.
\end{equation}
Therefore, for any $y\in\mathcal{B}_{0}$,
\begin{align*}
a(&x)D_{\nu}\mathcal{G}(x,y)\Big|^{-}_{\partial\mathcal{B}_{1}}
=D_{\nu}\mathcal{G}(x,y)\Big|^{-}_{\partial\mathcal{B}_{1}}\\
=&D_{\nu}\log|x-y|+\sum_{l=1}^{\infty}\Big[(\alpha\beta)^{l}
\Big(D_{\nu}\log|(\Phi_{1}\Phi_{2})^{l}(x)-y|+D_{\nu}\log|(\Phi_{2}\Phi_{1})^{l}(x)-y|\Big)\\
&-(\alpha\beta)^{l-1}\Big(\beta\,D_{\nu}\log|(\Phi_{2}\Phi_{1})^{l-1}\Phi_{2}(x)-y|
+\alpha\,D_{\nu}\log|(\Phi_{1}\Phi_{2})^{l-1}\Phi_{1}(x)-y|\Big)\Big]\\
=&D_{\nu}\log|x-y|+\sum_{l=1}^{\infty}\Big[(\alpha\beta)^{l}
\Big(D_{\nu}\log|(\Phi_{1}\Phi_{2})^{l}(x)-y|\\
&-D_{\nu}\log|(\Phi_{2}\Phi_{1})^{l-1}\Phi_{2}(x)-y|\Big)
-(\alpha\beta)^{l-1}\Big(\beta\,D_{\nu}\log|(\Phi_{2}\Phi_{1})^{l-1}\Phi_{2}(x)
-y|\\
&-\alpha\,D_{\nu}\log|(\Phi_{1}\Phi_{2})^{l-1}(x)-y|\Big)\Big]\\
=&(1+\alpha)\sum_{l=0}^{\infty}(\alpha\beta)^{l}
\Big(D_\nu\log|(\Phi_{1}\Phi_{2})^{l}(x)-y|-\beta D_\nu\log|(\Phi_{2}\Phi_{1})^{l}\Phi_{2}(x)-y|\Big)\\
=&k_{1}D_{\nu}\mathcal{G}(x,y)\Big|^{+}_{\partial\mathcal{B}_{1}}
=a(x)D_{\nu}\mathcal{G}(x,y)\Big|^{+}_{\partial\mathcal{B}_{1}},
\end{align*}
where we used $1+\alpha=2k_1/(k_{1}+1)$ in the last but one equality.

For the case when $y\in\mathcal{B}_{1}$, the singularity appears in $\mathcal{B}_{1}$. If $x\in\mathcal{B}_{0}$, then
$(\Phi_{2}\Phi_{1})^{l}(x)\in\mathcal{B}_{2}$ and $(\Phi_{2}\Phi_{1})^{l}\Phi_{2}(x)\in\mathcal{B}_{2}$. Therefore, $\mathcal{G}(x,y)$ is harmonic in $\mathcal{B}_{0}$. In the same way, $\mathcal{G}(x,y)$ is harmonic in $\mathcal{B}_{2}$ as well. When $x\in\mathcal{B}_{1}\setminus\{(1+\varepsilon/2,0)\}$, each term in the expression of $\mathcal{G}(x,y)$, with the exception of the term $\frac{1}{k_{1}}\log|x-y|$, is harmonic in $\mathcal{B}_{1}\setminus\{(1+\varepsilon/2,0)\}$ because $\Phi_1(x)\notin \cB_1$ and $(\Phi_2\Phi_1)^l\Phi_2(x)\in \cB_2$. Hence, when $x\in\mathcal{B}_{1}\setminus\{(1+\varepsilon/2,0)\}$, $k_{1}\Delta_{x}\mathcal{G}(x,y)=\delta(x-y)$.

Now we verify the continuity of $\mathcal{G}(x,y)$ and $a(x)D_{\nu}\mathcal{G}(x,y)$ across the circle $\partial \cB_1$. The proof of the continuity across $\partial \cB_2$ is similar. For $x\in\partial\mathcal{B}_{1}$, using \eqref{phi1x=x} again, we have, for any $y\in\mathcal{B}_{1}$,
\begin{align*}
\mathcal{G}(&x,y)\Big|_{\partial\mathcal{B}_{1}}^{-}\\
=&{\frac 2 {k_{1}+1}}\sum_{l=0}^{\infty}(\alpha\beta)^{l}
\Big(\log|(\Phi_{2}\Phi_{1})^{l}(x)-y|
-\beta\log|(\Phi_{2}\Phi_{1})^{l}\Phi_{2}(x)-y|\Big)\\
=&\frac 2 {k_{1}+1}(\log|x-y|-\beta\log|\Phi_{2}(x)-y|)\\
&+\frac 2 {k_{1}+1}\sum_{l=1}^{\infty}(\alpha\beta)^{l}
\Big(\log|(\Phi_{2}\Phi_{1})^{l-1}\Phi_{2}(x)-y|
-\beta\log|(\Phi_{2}\Phi_{1})^{l}\Phi_{2}(x)-y|\Big)\\
=&\frac{1+\alpha}{k_{1}}\log|x-y|-\frac{2\beta(1-\alpha)}
{(k_{1}+1)}\sum_{l=0}^{\infty}(\alpha\beta)^{l}
\log|(\Phi_{2}\Phi_{1})^{l}\Phi_{2}(x)-y|\\
=&{\frac 1 {k_1}}\big(\log|x-y|+\alpha\log|\Phi_{1}(x)-y|\big)
-{\frac{4\beta}{(k_{1}+1)^2}}\sum_{l=0}^{\infty}(\alpha\beta)^{l}
\log|(\Phi_{2}\Phi_{1})^{l}\Phi_{2}(x)-y|\\
=&\mathcal{G}(x,y)\Big|_{\partial\mathcal{B}_{1}}^{+}.
\end{align*}

Next, we check the continuity of $a(x)D_{\nu}\mathcal{G}(x,y)$ on $\partial\mathcal{B}_{1}$. In view of \eqref{normalB1} again, for any $y\in\mathcal{B}_{1}$,
\begin{align*}
a(&x)D_{\nu}\mathcal{G}(x,y)\Big|^{-}_{\partial\mathcal{B}_{1}}=D_{\nu}\mathcal{G}(x,y)\Big|^{-}_{\partial\mathcal{B}_{1}}\\
=&{\frac 2 {k_{1}+1}}\sum_{l=0}^{\infty}(\alpha\beta)^{l}
\Big(D_{\nu}\log|(\Phi_{2}\Phi_{1})^{l}(x)-y|
-\beta\,D_{\nu}\log|(\Phi_{2}\Phi_{1})^{l}\Phi_{2}(x)-y|\Big)\\
=&{\frac 2 {k_{1}+1}}
\Big(D_{\nu}\log|x-y|
-\beta\,D_{\nu}\log|\Phi_{2}(x)-y|\Big)\\
&-{\frac 2 {k_{1}+1}}\sum_{l=1}^{\infty}(\alpha\beta)^{l}
\Big(D_{\nu}\log|(\Phi_{2}\Phi_{1})^{l-1}\Phi_{2}(x)-y|
+\beta\,D_{\nu}\log|(\Phi_{2}\Phi_{1})^{l}\Phi_{2}(x)-y|\Big)\\
=&(1-\alpha)D_{\nu}\log|x-y|-\frac 2 {k_{1}+1}\beta\,D_{\nu}\log|\Phi_{2}(x)-y|\\
&-\frac 2 {k_{1}+1}\sum_{l=1}^{\infty}(\alpha\beta)^{l}
\Big(D_{\nu}\log|(\Phi_{2}\Phi_{1})^{l-1}\Phi_{2}(x)-y|
+\beta\,D_{\nu}\log|(\Phi_{2}\Phi_{1})^{l}\Phi_{2}(x)-y|\Big)\\
=&\big(D_{\nu}\log|x-y|+\alpha\,D_{\nu}\log|\Phi_{1}(x)-y|\big)\\
&-{\frac{4\beta\,k_{1}}{(k_{1}+1)^2}}\sum_{l=0}^{\infty}(\alpha\beta)^{l}
D_{\nu}\log|(\Phi_{2}\Phi_{1})^{l}\Phi_{2}(x)-y|\\
=&k_{1}D_{\nu}\mathcal{G}(x,y)\Big|^{+}_{\partial\mathcal{B}_{1}}=a(x)D_{\nu}\mathcal{G}(x,y)\Big|^{+}_{\partial\mathcal{B}_{1}},
\end{align*}
where we used $1+\alpha=2k_1/(k_{1}+1)$ in the last but one equality.

The case when $y\in\mathcal{B}_{2}$ is similar, and thus omitted. The proof is finished.
\end{proof}

With Lemma \ref{lem1}, we are ready to construct a Green's function $G(x,y)$ of the operator $L$. Define
$$
G(x,y)=
\begin{cases}
\mathcal{G}(x,y)\quad&\text{for}\,y\in\mathcal{B}_{0},\\
\mathcal{G}(x,y)+\frac{\alpha}{1-\alpha}\mathcal{G}(x,(1+\varepsilon/2,0))
\quad&\text{for}\,y\in\mathcal{B}_{1},\\
\mathcal{G}(x,y)+\frac{\beta}{1-\beta}\mathcal{G}(x,(-1-\varepsilon/2,0))
\quad&\text{for}\,y\in\mathcal{B}_{2}.
\end{cases}
$$
When defining $\mathcal{G}$, $(1+\varepsilon/2,0)$ and $(-1-\varepsilon/2,0)$ are removed, since $\Phi_{1}$ and $\Phi_{2}$ have singularity at these points. Nevertheless, in the following proposition we prove that for fixed $y$, $G(x,y)$ is well-defined in $\mathbb{R}^{2}\setminus\{y\}$. In particular,
$$
\lim_{x\rightarrow(1+\varepsilon/2,0)}G(x,y)
\quad\big(\mbox{or}\quad\lim_{x\rightarrow(-1-\varepsilon/2,0)}G(x,y)\big)
$$
exists when $y\neq (1+\varepsilon/2,0)$ (or $y\neq (-1-\varepsilon/2,0)$, respectively).

\begin{prop}\label{prop1}
The function $G(x,y)$ defined above is a Green's function of $L$.
\end{prop}

\begin{proof}
From Lemma \ref{lem1}, $\mathcal{G}(x,y)$ satisfies the compatibility conditions. Thus by linearity, $G(x,y)$ satisfies the compatibility conditions as well. It remains to prove
$$
a(x)\Delta_{x}G(x,y)=\delta(x-y)\quad\mbox{for}
~~x\notin\partial\mathcal{B}_{1}\cup\partial\mathcal{B}_{2}.
$$

For $y\in\mathcal{B}_{0}$, this is proved in Lemma \ref{lem1}. It remains to treat the case when $y\in\mathcal{B}_{1}\cup\mathcal{B}_{2}$. We only consider the case when $y\in\mathcal{B}_{1}$ because the case when $y\in\mathcal{B}_{2}$ is similar. As we mentioned in Remark \ref{rem1}, $\Phi_{1}$ has a singularity at $(1+\varepsilon/2,0)$. By the definition of $\Phi_{1}$, \eqref{Phi1real}, we have, for fixed $y\neq(1+\varepsilon/2,0)$, when $x$ is near $(1+\varepsilon/2,0)$,
$$
\log|\Phi_{1}(x)-y|\sim-\log|x-(1+\varepsilon/2,0)|,
$$
which implies that
$$
k_1\mathcal{G}(x,y)\sim-\alpha\log|x-(1+\varepsilon/2,0)|.
$$
Moreover, the singular part in $k_1\mathcal{G}(x,(1+\varepsilon/2,0))$ behaves like
$$
(1-\alpha)\log|x-(1+\varepsilon/2,0)|.
$$
Therefore, $G(x,y)$ is bounded around $(1+\varepsilon/2,0)$. This yields that $(1+\varepsilon/2,0)$ is a removable singularity of $G(\cdot,y)$ for $y\neq(1+\varepsilon/2,0)$, so that
$$
k_1\Delta_{x}G(x,y)=\delta(x-y).
$$

When $y=(1+\varepsilon/2,0)$, it suffices to consider $x\in\mathcal{B}_{1}$. Note that
$$
\log|\Phi_1(x)-y|=-\log|x-y|\quad
\text{and}\quad
G(x,y)=\frac 1 {1-\alpha}\cdot\cG(x,y).
$$
Thus, by the definition of $\cG(x,y)$,
$k_1G(x,y)-\log|x-y|$
is harmonic in $\mathcal{B}_{1}$.
Therefore, $G(x,y)$ is well defined and
$$
k_1\Delta_{x}G(x,y)=\delta(x-y).
$$
The proof is finished.
\end{proof}

\section{Derivative estimates}\label{sec_thm12}

In order to establish the derivative estimates for solutions of \eqref{mainequ}, we first derive derivative estimates of the maps $\Phi_{1}\Phi_{2}$, $\Phi_{2}\Phi_{1}$, and their compositions.

\begin{lemma}\label{lem_Phi_bound}
For any integers $l\ge 0$ and $m\ge 1$,
\begin{equation}
                        \label{Phi21bound}
|D^{m}(\Phi_{2}\Phi_{1})^{l}(x)|\le \frac{Cl^{m-1}}{(1+2\sqrt{\varepsilon})^{2l}}\quad\text{in}\,\,\mathcal{B}_{2},
\end{equation}
and
\begin{equation*}
|D^{m}(\Phi_{1}\Phi_{2})^{l}(x)|\le \frac{Cl^{m-1}}{(1+2\sqrt{\varepsilon})^{2l}}\quad\text{in}\,\,\mathcal{B}_{1},
\end{equation*}
where $C$ depends only on $m$.
\end{lemma}

\begin{proof}
We identify a point $x=(x_1,x_2)\in \bR^2$ with a complex number $z=x_1+i x_2\in\mathbb{C}$. For convenience, here we take $a=1+\varepsilon/2$.
With respect to the complex variable,
\begin{equation}\label{def_Phi_complex}
\Phi_1(z)=\frac{a\bar z-(a^2-1)}{\bar z-a}\quad\text{and}\quad
\Phi_2(z)=\frac{-a\bar z-(a^2-1)}{\bar z+a}.
\end{equation}
Thus,
$$
(\Phi_{2}\Phi_{1})(z):=\Phi_{2}\circ\Phi_{1}(z)=\frac{-(2a^2-1)z+2a(a^2-1)}{2az-(2a^2-1)}.
$$
Using a translation and dilation of coordinates
\begin{equation*}
2az-(2a^2-1)\to z,\quad\mbox{ and}\quad 2a(\Phi_{2}\Phi_{1})(z)-(2a^2-1)\to (\Phi_{2}\Phi_{1})(z),
\end{equation*}
we obtain
$$
(\Phi_{2}\Phi_{1})(z):=-1/z-2(2a^2-1).
$$
To find an expression of $(\Phi_{2}\Phi_{1})^l$, we consider the fixed points of $\Phi_{2}\Phi_{1}$.
Notice that $\Phi_{2}\Phi_{1}$ has two fixed points, the one in $\mathcal{B}_{1}$ given by
$$
\lambda_{1}:=-(2a^2-1)+2a\sqrt{a^2-1}\sim -1+2\sqrt{\varepsilon},
$$
and the one in $\mathcal{B}_{2}$ given by
$$
\lambda_{2}:=-(2a^2-1)-2a\sqrt{a^2-1}\sim -1-2\sqrt{\varepsilon}.
$$

Clearly, for $z\in\mathcal{B}_{2}$,
$$
(\Phi_{2}\Phi_{1})(z)-\lambda_{2}=-1/z-2(2a^2-1)-\lambda_{2}
=-1/z+1/\lambda_{2}=\frac{z-\lambda_{2}}{z\lambda_{2}},
$$
which implies that for any $z\neq \lambda_{2}$,
$$
\frac 1 {(\Phi_{2}\Phi_{1})(z)-\lambda_{2}}=\frac{\lambda_{2}^2}{z-\lambda_{2}}+\lambda_{2},
$$
and thus
$$
\frac 1 {(\Phi_{2}\Phi_{1})(z)-\lambda_{2}}-\frac{\lambda_{2}}{1-\lambda_{2}^2}=
\lambda_{2}^2\Big(\frac{1}{z-\lambda_{2}}-\frac{\lambda_{2}}{1-\lambda_{2}^2}\Big).
$$
By iteration, we have for any $z\neq \lambda_{2}$ and $l\ge 0$,
$$
\frac 1 {(\Phi_{2}\Phi_{1})^{l}(z)-\lambda_{2}}=
\lambda_{2}^{2l}\Big(\frac{1}{z-\lambda_{2}}-\frac{\lambda_{2}}{1-\lambda_{2}^2}\Big)
+\frac{\lambda_{2}}{1-\lambda_{2}^2}.
$$
Hence
\begin{align}
(\Phi_{2}\Phi_{1})^{l}(z)&=\lambda_{2}+
\Big(\lambda_{2}^{2l}\big(\frac{1}{z-\lambda_{2}}-\frac{\lambda_{2}}{1-\lambda_{2}^2}\big)
+\frac{\lambda_{2}}{1-\lambda_{2}^2}\Big)^{-1}\nonumber\\
&=\lambda_{2}+\frac{\lambda_{2}^2-1}{\lambda_{2}^{2l+1}}\cdot \frac{z-\lambda_{2}}{z(1-\lambda_{2}^{-2l})
-(\lambda_{2}^{-1}-\lambda_{2}^{-2l+1})}\nonumber\\
                            \label{eq11.37}
&=\lambda_{2}+\frac{\lambda_{2}^2-1}{\lambda_{2}^{2l+1}}\cdot \frac{1}{1-\lambda_{2}^{-2l}}\Big(1+\frac{\lambda_{2}^{-1}-\lambda_{2}}
{z(1-\lambda_{2}^{-2l})-(\lambda_{2}^{-1}-\lambda_{2}^{-2l+1})}\Big).
\end{align}
Note that the identity above also holds when $z=\lambda_2$.

Next, we differentiate \eqref{eq11.37} with respect to $z$. For $m=1$,
\begin{align}\label{firstderivative}
D(\Phi_{2}\Phi_{1})^{l}(z)&=\frac{(\lambda_{2}-1/\lambda_2)^2}
{\lambda_{2}^{2l}}\cdot \Big(z(1-\lambda_{2}^{-2l})-(\lambda_2^{-1}
-\lambda_{2}^{-2l+1})\Big)^{-2}\nonumber\\
&=\frac{(\lambda_{2}-1/\lambda_2)^2}{\lambda_{2}^{2l}}\cdot
\Big((z-\lambda_2^{-1})(1-\lambda_{2}^{-2l})
+(\lambda_{2}-\lambda_2^{-1})\lambda_{2}^{-2l}\Big)^{-2}.
\end{align}
Since
$\lambda_{2}-\lambda_{2}^{-1}\sim -\sqrt\varepsilon$ and, for $z\in \mathcal{B}_{2}$ (in the new coordinates),
$$
\mathrm{Re}~ (z-\lambda_2^{-1})\le -2a(a-1)-(2a^2-1)-\lambda_2^{-1}\le -1-\lambda_2^{-1}\lesssim -\sqrt\varepsilon,
$$
it follows that
\begin{equation}\label{denominator}
\Big|(z-\lambda_2^{-1})(1-\lambda_{2}^{-2l})
+(\lambda_{2}-\lambda_2^{-1})\lambda_{2}^{-2l}\Big|\gtrsim \sqrt\varepsilon.
\end{equation}
Thus, we obtain for $l\geq0$,
$$
|D(\Phi_{2}\Phi_{1})^{l}(z)|\le \frac{C}{\lambda_{2}^{2l}}\le \frac{C}{(1+2\sqrt{\varepsilon})^{2l}}.
$$

For higher-order derivatives, from \eqref{firstderivative}, we have for  $m\geq2$,
\begin{align*}
&D^{m}(\Phi_{2}\Phi_{1})^{l}(z)\\
&=\frac{(\lambda_{2}-1/\lambda_2)^2}
{\lambda_{2}^{2l}}(-1)^{m-1}m!(1-\lambda_{2}^{-2l})^{m-1} \Big(z(1-\lambda_{2}^{-2l})-(\lambda_2^{-1}-\lambda_{2}^{-2l+1})\Big)^{-(m+1)}.
\end{align*}
For $z\in\mathcal{B}_{2}$, using \eqref{denominator} and the simple inequality
$$
0\le 1-\lambda_2^{-2l}\le Cl\sqrt\varepsilon,
$$
we get
\begin{equation}\label{Phi21m}
D^{m}(\Phi_{2}\Phi_{1})^{l}(z)\le \frac{Cl^{m-1}}{\lambda_{2}^{2l}}\le \frac{Cl^{m-1}}{(1+2\sqrt{\varepsilon})^{2l}}.
\end{equation}
Thus, \eqref{Phi21bound} is proved.

Similarly, write
$$(\Phi_{1}\Phi_{2})(z):=\Phi_1\circ\Phi_2(z)=\frac{(2a^2-1)z+2a(a^2-1)}{2az+2a^2-1}.
$$
By a translation and dilation of coordinates
$$2az+2a^2-1\to z\quad\mbox{ and}\quad 2a(\Phi_{1}\Phi_{2})(z)+2a^2-1\to (\Phi_{1}\Phi_{2})(z),$$ we obtain
$$
(\Phi_{1}\Phi_{2})(z):=-1/z+2(2a^2-1).
$$
The map $\Phi_{1}\Phi_{2}$ also has two fixed points, the one in $\mathcal{B}_{1}$ given by
$$
\tilde{\lambda}_{1}:=2a^2-1+2a\sqrt{a^2-1}\sim 1+2\sqrt{\varepsilon},
$$
and the one in $\mathcal{B}_{2}$ given by
$$\tilde{\lambda}_{2}:=2a^2-1-2a\sqrt{a^2-1}\sim 1-2\sqrt{\varepsilon}.$$

For $z\in\mathcal{B}_{1}$,
$$
(\Phi_{1}\Phi_{2})(z)-\tilde{\lambda}_{1}=-1/z+2(2a^2-1)-\tilde{\lambda}_{1}
=-1/z+1/\tilde{\lambda}_{1}=\frac{z-\tilde{\lambda}_{1}}{z\tilde{\lambda}_{1}}.
$$
In the same way,
\begin{align}\label{firstderivative12}
D(\Phi_{1}\Phi_{2})^{l}(z)
&=\frac{(\tilde{\lambda}_{1}-1/\tilde{\lambda}_{1})^2}
{\tilde{\lambda}_{1}^{2l}}\cdot \Big(z(1-\tilde{\lambda}_{1}^{-2l})-(\tilde{\lambda}_{1}^{-1}
-\tilde{\lambda}_{1}^{-2l+1})\Big)^{-2}\nonumber\\
&=\frac{(\tilde{\lambda}_{1}-1/\tilde{\lambda}_{1})^2}
{\tilde{\lambda}_{1}^{2l}}\cdot
\Big((z-\tilde{\lambda}_{1}^{-1})(1-\tilde{\lambda}_{1}^{-2l})
+(\tilde{\lambda}_{1}-\tilde{\lambda}_{1}^{-1})
\tilde{\lambda}_{1}^{-2l}\Big)^{-2}.
\end{align}
Since
$\tilde{\lambda}_{1}-\tilde{\lambda}_{1}^{-1}\sim \sqrt\varepsilon$ and for $z\in \mathcal{B}_1$ (in the new coordinates),
$$
\mathrm{Re}~ (z-\tilde{\lambda}_{1}^{-1})\ge 2a(a-1)+2a^2-1-\tilde{\lambda}_{1}^{-1}\ge 1-\tilde{\lambda}_{1}^{-1}\gtrsim \sqrt\varepsilon,
$$
we have
\begin{equation*}
\Big|(z-\tilde{\lambda}_{1}^{-1})(1-\tilde{\lambda}_{1}^{-2l})
+(\tilde{\lambda}_{1}-\tilde{\lambda}_{1}^{-1})
\tilde{\lambda}_{1}^{-2l}\Big|\gtrsim \sqrt\varepsilon.
\end{equation*}
Thus, for $l\geq 0$,
$$
|D(\Phi_{1}\Phi_{2})^{l}(z)|\le \frac{C}{\tilde{\lambda}_{1}^{2l}}\le \frac{C}{(1+2\sqrt{\varepsilon})^{2l}}.
$$
For higher-order derivatives, using
$$
0\le 1-\tilde{\lambda}_{1}^{-2l}\le Cl\sqrt\varepsilon,
$$
we obtain
\begin{equation*}
D^{m}(\Phi_{1}\Phi_{2})^{l}(z)\le \frac{Cl^{m-1}}{\tilde{\lambda}_{1}^{2l}}\le \frac{Cl^{m-1}}{(1+2\sqrt{\varepsilon})^{2l}}.
\end{equation*}
This completes the proof of the lemma.
\end{proof}

In the next lemma, we show a Schauder estimate for the elliptic equation \eqref{mainequ} with piecewise constant coefficient in the upper and lower half balls. We allow the coefficient to be partially degenerate and we give an explicit dependence of the constant in the estimate with respect to the coefficient.
\begin{lemma}\label{lem2}
Let $m\ge 1$, $\gamma\in (0,1)$, and $k\in (0,\infty)$ be constants. Suppose that $u\in W^1_2(B_1)$ is a weak solution to $D_i(a(x)D_i u)=D_if_i$, where $a(x)=k\chi_{B_1^+}+\chi_{B_1^-}$ and $f_i$ is $C^{m-1,\gamma}$ in $B_1^+$ and $B_1^-$. Then $u$ is $C^{m,\gamma}$ in $B_1^+$ and $B_1^-$, and
\begin{align*}
&\|u\|_{C^{m,\gamma}(B_{1/2}^+)}+\|u\|_{C^{m,\gamma}(B_{1/2}^-)}\\
&\le C \|u\|_{L_2(B_1)}+Ck^{-1}\|f\|_{C^{m-1,\gamma}(B_1^+)}
+C\|f\|_{C^{m-1,\gamma}(B_1^-)},
\end{align*}
where $C$ is a constant depending only on $d$ and $m$.
\end{lemma}
\begin{proof}
We define $v(x_1,x_2)=u(-x_1,x_2)$ on $B_1^+$. Then when $k\ge 1$, $u$ and $v$ satisfy the elliptic system
$$
\Delta u=f/k,\quad \Delta v=f(-x_1,x_2)\quad\text{on}\quad B_1^+
$$
with the boundary conditions
$$
u-v=0\quad \text{and}\quad D_1 u+k^{-1}D_1 v=0\quad\text{on}\quad B_1\cap\{x_1=0\}.
$$
When $k\in (0,1)$, the boundary conditions can be written as
$$
u-v=0\quad \text{and}\quad kD_1 u+D_1 v=0 \quad\text{on}\quad B_1\cap\{x_1=0\}.
$$
In both cases, the boundary conditions are complementing boundary conditions introduced in \cite{adn} with uniformly bounded (constant) coefficients. Therefore, the desired estimate follows immediately from the classical regularity theory for elliptic systems. See, for instance, \cite{adn}.
\end{proof}

\subsection{Proof of Theorem \ref{mainthm1}}
With the help of Green's function constructed in Section \ref{sec_Green}, and Lemmas \ref{lem_Phi_bound} and \ref{lem2}, we are in the position to consider the non-homogeneous equation \eqref{mainequ} in $B_1$.

\begin{proof}[Proof of Theorem \ref{mainthm1}] By dividing $u$ and $f$ by $C_1$, without loss of generality, we may assume that $C_1=1$.
We take a cutoff function $\eta\in{C}_{0}^{\infty}(B_{1})$ such that $\eta=1$ in $B_{1/2}$. Let $v=u\eta$, which satisfies
\begin{equation}\label{equ410}
D_{i}(aD_{i}v(x))=D_{i}\tilde{f}_{i}+\tilde f_3\quad\mbox{in}~~\mathbb{R}^{2},
\end{equation}
where
$$
\tilde{f}_{i}=f_{i}\eta+auD_{i}\eta,\quad \tilde f_3=-f_{i}D_{i}\eta+aD_{i}uD_{i}\eta.
$$
Since $\mathrm{supp}(D_{i}\eta)\subset\,B_{1}\setminus\overline{B}_{1/2}$, by \eqref{eq11.56} and Lemma \ref{lem2} with a conformal map, we infer that $\tilde{f}_{1}$, $\tilde{f}_{2}$, and $\tilde{f}_{3}$ are compactly supported in $B_1$, piecewise $C^{\gamma}$, and
\begin{equation}
                            \label{eq11.56b}
\|\tilde f_i\|_{C^\gamma(\cB_j)}\le Ck_j,\,\,j=0,1,2,~~i=1,2,3.
\end{equation}

Now define
\begin{align}\label{def_tildeu}
\tilde u(x)&=-\int_{\mathcal{B}_1}D_{y_i}G (x,y)\tilde f_i(y)\,dy-\int_{\mathcal{B}_2}D_{y_i}G (x,y)\tilde f_i(y)\,dy\nonumber\\
&\quad -\int_{\mathcal{B}_0}D_{y_i}G (x,y)\tilde f_i(y)\,dy+\int_{B_1}G(x,y)\tilde f_3(y)\,dy\nonumber\\
&:=-w_1(x)-w_2(x)-w_0(x)+w_3(x),
\end{align}
which is a solution to \eqref{equ410} in $\bR^2$.

{\bf Claim:} We have $v=\tilde u+C_0$ for some constant $C_0$.

Assuming for the moment that the claim above is proved, it suffices for us to estimate $\tilde{u}$ in $B_{1/2}$.

{\bf Gradient estimates.}
To estimate $\tilde u$, we define for $j=0,1,2$,
\begin{equation}
                            \label{eq3.12}
h_{j}(x)=\int_{\mathcal{B}_{j}}D_{y_{i}}\log|x-y|\tilde f_i(y)\,dy
\end{equation}
and
\begin{equation}
                            \label{eq12.22}
g_{j}(x)=\int_{\mathcal{B}_{j}}\log|x-y|\tilde f_3(y)\,dy.
\end{equation}
Since $\log|x-y|$ is the fundamental solution of the Laplace equation in $\mathbb{R}^{2}$, $h_{j}$ and $g_j$ satisfy
$$
\Delta\,h_{j}=-D_{i}(\tilde f_i\chi_{\mathcal{B}_{j}})
\quad \text{and}\quad \Delta\,g_{j}=\tilde f_3\chi_{\mathcal{B}_{j}}
\quad\text{in}\,\,\mathbb{R}^{2}.
$$
Because $\tilde f_i$ is piecewise $C^\gamma$ and the interfaces $\partial\mathcal{B}_{1}$ and $\partial\mathcal{B}_{2}$ are smooth, using Lemma \ref{lem2} we see that for $j=1,2$, $h_{j}$ and $g_j$ are piecewise $C^{1,\gamma}$. Moreover, due to \eqref{eq11.56b},
\begin{equation}
                        \label{eq3.28}
\|h_j\|_{C^{1,\gamma}(B_{3})}+
\|g_j\|_{C^{1,\gamma}(B_{3})}\le C k_j.
\end{equation}
Using Lemma 2.1 in \cite{dz} and the same argument for $g$ as in the proof of Theorem 4.12 there, we see that $h_{0}$ and $g_0$ are also piecewise $C^{1,\gamma}$, and
\begin{equation}
                        \label{eq3.33}
\|h_0\|_{C^{1,\gamma}(B_{3})}+
\|g_0\|_{C^{1,\gamma}(B_{3})}\le C.
\end{equation}

\underline{Estimates in $\mathcal{B}_{0}\cap\,B_{1/2}$:}
First we consider the narrow region between $\mathcal{B}_{1}$ and $\mathcal{B}_{2}$.
For $x\in\mathcal{B}_{0}\cap\,B_{1/2}$, by the definition of $G(x,y)$, we have
\begin{align}\label{def_w1}
w_{1}(x)&=\frac{2}{k_{1}+1}\sum_{l=0}^{\infty}(\alpha\beta)^{l}
\Big(\int_{\mathcal{B}_{1}}D_{y_{i}}\log|(\Phi_{2}\Phi_{1})^{l}(x)-y|\tilde f_i(y)\,dy\nonumber\\
&\qquad-\beta\int_{\mathcal{B}_{1}}D_{y_{i}}\log|(\Phi_{2}\Phi_{1})^{l}\Phi_{2}(x)-y|\tilde f_i(y)\,dy\Big)\nonumber\\
&=\frac{2}{k_{1}+1}\sum_{l=0}^{\infty}(\alpha\beta)^{l}
\Big(h_{1}((\Phi_{2}\Phi_{1})^{l}(x))-\beta\,h_{1}
((\Phi_{2}\Phi_{1})^{l}\Phi_{2}(x))\Big).
\end{align}

By the definition \eqref{def_Phi_complex}, $\|D(\Phi_{2}\Phi_{1})\|_{L^{\infty}(\mathcal{B}_{0})}$ and $\|D\Phi_{2}\|_{L^{\infty}(\mathcal{B}_{0})}$ are bounded by $1$.
For $x\in\mathcal{B}_{0}\cap\,B_{1/2}$, $\Phi_{2}\Phi_{1}(x)$ and $\Phi_{2}(x)$ are both in $\mathcal{B}_{2}$. Using the chain rule and \eqref{def_w1}, we have
\begin{align*}
&|Dw_{1}(x)|=\frac{2}{k_{1}+1}\sum_{l=0}^{\infty}|\alpha\beta|^{l}
\left|\Big(Dh_{1}((\Phi_{2}\Phi_{1})^{l}(x))-\beta\,Dh_{1}((\Phi_{2}\Phi_{1})^{l}
\Phi_{2}(x))\Big)\right|\\
&\le\frac{2}{k_{1}+1}\Big(\|Dh_{1}\|_{L^{\infty}(\mathcal{B}_{0})}+|\beta|
\|Dh_{1}\|_{L^{\infty}(\mathcal{B}_{2})}
\|D\Phi_{2}\|_{L^{\infty}(\mathcal{B}_{0})}\Big)\\
&\quad+\frac{2}{k_{1}+1}\sum_{l=1}^{\infty}|\alpha\beta|^{l}\Big(\|Dh_{1}\|_{L^{\infty}(\mathcal{B}_{2})}\|D(\Phi_{2}\Phi_{1})^{l-1}\|_{L^{\infty}(\mathcal{B}_{2})}\|D(\Phi_{2}\Phi_{1})\|_{L^{\infty}(\mathcal{B}_{0})}\\
&\quad+|\beta|\|Dh_{1}\|_{L^{\infty}(\mathcal{B}_{2})}\|D(\Phi_{2}\Phi_{1})^{l}\|_{L^{\infty}(\mathcal{B}_{2})}\|D\Phi_{2}\|_{L^{\infty}(\mathcal{B}_{0})}
\Big),
\end{align*}
which, thanks to Lemma \ref{lem_Phi_bound} and \eqref{eq3.28}, is bounded by
\begin{align*}
&\le\frac{Ck_1}{k_{1}+1}\Big(1+
\sum_{l=0}^{\infty}|\alpha\beta|^{l}
\|D(\Phi_{2}\Phi_{1})^{l}\|_{L^{\infty}(\mathcal{B}_{2})}\Big)\\
&\le\,\frac{Ck_1}{k_{1}+1}\left(1+
\frac{|\alpha\beta|}{1-\frac{|\alpha\beta|}{(1+2\sqrt{\varepsilon})^{2}}}\right)\\
&\le \frac{Ck_1}{k_{1}+1}\cdot
\frac{1}{1-(1-\sqrt\varepsilon)|\alpha\beta|}.
\end{align*}
Similarly, by the definition of $G(x,y)$,
\begin{align*}
w_{2}(x)&=(1+\beta)\sum_{l=0}^{\infty}(\alpha\beta)^{l}
\Big(\int_{\mathcal{B}_{2}}D_{y_{i}}\log|(\Phi_{1}\Phi_{2})^{l}(x)-y|\tilde f_i(y)\,dy\\
&\qquad-\alpha\int_{\mathcal{B}_{2}}D_{y_{i}}\log|(\Phi_{1}\Phi_{2})^{l}\Phi_{1}(x)-y|\tilde f_i(y)\,dy\Big),\\
&=\frac{2}{k_{2}+1}\sum_{l=0}^{\infty}(\alpha\beta)^{l}\Big(h_{2}((\Phi_{1}\Phi_{2})^{l}(x))-\alpha\,h_{2}((\Phi_{1}\Phi_{2})^{l}\Phi_{1}(x))\Big).
\end{align*}
Since for $x\in\mathcal{B}_{0}\cap\,B_{1/2}$, $(\Phi_{1}\Phi_{2})(x)$ and $\Phi_{1}(x)$, $l\geq1$ are both in $\mathcal{B}_{1}$, using Lemma \ref{lem_Phi_bound} and \eqref{eq3.28},
\begin{align*}
|Dw_{2}(x)|&=\frac{2}{k_{2}+1}\sum_{l=0}^{\infty}|\alpha\beta|^{l}
\left|\Big(Dh_{2}((\Phi_{1}\Phi_{2})^{l}(x))
-\alpha\,Dh_{2}((\Phi_{1}\Phi_{2})^{l}\Phi_{1}(x))\Big)\right|\\
&\le\frac{C k_2}{k_{2}+1}\Big(1
+\sum_{l=0}^{\infty}|\alpha\beta|^{l}
\|D(\Phi_{1}\Phi_{2})^{l}\|_{L^{\infty}(\mathcal{B}_{1})}\Big)\\
&\le \frac{Ck_2}{k_{2}+1}\cdot\frac{1}
{1-(1-\sqrt\varepsilon)|\alpha\beta|}.
\end{align*}
Because
\begin{align*}
&w_{0}(x)\\
&=\int_{\mathcal{B}_{0}}D_{y_{i}}\log|x-y|\tilde f_i(y)\,dy
+\sum_{l=1}^{\infty}\Big[(\alpha\beta)^{l}
\Big(\int_{\mathcal{B}_{0}}D_{y_{i}}\log|(\Phi_{1}\Phi_{2})^{l}(x)-y|\tilde f_i(y)\,dy\\
&\quad+\int_{\mathcal{B}_{0}}D_{y_{i}}\log|(\Phi_{2}\Phi_{1})^{l}(x)
-y|\tilde f_i(y)\,dy\Big)\\
&\quad -(\alpha\beta)^{l-1}\Big(\beta\int_{\mathcal{B}_{0}}D_{y_{i}}
\log|(\Phi_{2}\Phi_{1})^{l-1}\Phi_{2}(x)-y|\tilde f_i(y)\,dy\\
&\quad+\alpha\int_{\mathcal{B}_{0}}D_{y_{i}}\log|(\Phi_{1}\Phi_{2})^{l-1}
\Phi_{1}(x)-y|\tilde f_i(y)\,dy\Big)\Big]\\
&=h_{0}(x)+\sum_{l=1}^{\infty}\Big[(\alpha\beta)^{l}
\Big(h_{0}((\Phi_{1}\Phi_{2})^{l}(x))+h_{0}((\Phi_{2}\Phi_{1})^{l}(x))\Big)\\
&\quad-(\alpha\beta)^{l-1}\Big(\beta\,h_{0}((\Phi_{2}\Phi_{1})^{l-1}\Phi_{2}(x))
+\alpha\,h_{0}((\Phi_{1}\Phi_{2})^{l-1}\Phi_{1}(x))\Big)\Big],
\end{align*}
using the same argument as above, from Lemma \ref{lem_Phi_bound} and \eqref{eq3.33}, we have
\begin{align*}
|Dw_{0}(x)|&\le\,C
+C\sum_{l=0}^{\infty}|\alpha\beta|^{l}\Big(\|D(\Phi_{1}\Phi_{2})^{l}\|_{L^{\infty}(\mathcal{B}_{1})}+\|D(\Phi_{2}\Phi_{1})^{l}\|_{L^{\infty}(\mathcal{B}_{2})}\Big)\\
&\le\frac{C}{1-(1-\sqrt\varepsilon)|\alpha\beta|}.
\end{align*}

The estimate of $w_3$ is also similar by using \eqref{eq12.22}, \eqref{eq3.28}, and \eqref{eq3.33}.
Therefore, recalling $\tilde{u}=-w_{1}-w_{2}-w_{0}+w_3$, we have, for $x\in\mathcal{B}_{0}\cap\,B_{1/2}$,
\begin{equation}
                                    \label{eq4.24}
|D\tilde u|\le\frac{C}{1-(1-\sqrt\varepsilon)|\alpha\beta|}.
\end{equation}

\underline{Estimates in $\mathcal{B}_{1}\cap\,B_{1/2}$:}
In this case, we have $(\Phi_{1}\Phi_{2})(x)\in\mathcal{B}_{1}$ and $\Phi_2(x)\in \cB_2$. By the definition \eqref{def_tildeu} and the same argument as in the case $x\in\mathcal{B}_{0}\cap\,B_{1/2}$, we have
\begin{align*}
w_{1}(x)=&\int_{\mathcal{B}_1}D_{y_i}G (x,y)\tilde f_i(y)\,dy\\
=&\frac{1}{k_{1}}\int_{\mathcal{B}_1}D_{y_i}\log|x-y|\tilde f_i(y)\,dy+\frac{\alpha}{k_{1}}\int_{\mathcal{B}_1}D_{y_i}\log|\Phi_{1}(x)-y|\tilde f_i(y)\,dy\\
&-{\frac{4\beta}{(k_{1}+1)^2}}\sum_{l=0}^{\infty}(\alpha\beta)^{l}
\int_{\mathcal{B}_1}D_{y_i}\log|(\Phi_{2}\Phi_{1})^{l}\Phi_{2}(x)-y|\tilde f_i(y)\,dy\\
=&\frac{1}{k_{1}}h_{1}(x)+\frac{\alpha}{k_{1}}h_{1}(\Phi_{1}(x))
-{\frac{4\beta}{(k_{1}+1)^2}}\sum_{l=0}^{\infty}(\alpha\beta)^{l}h_{1}((\Phi_{2}\Phi_{1})^{l}\Phi_{2}(x)).
\end{align*}
Thus, using Lemma \ref{lem_Phi_bound} and \eqref{eq3.28},
$$
|Dw_{1}(x)|\le\,C+\frac{C}{(k_{1}+1)^{2}}\cdot\frac{1}
{1-(1-\sqrt\varepsilon)|\alpha\beta|}\qquad\mbox{in}\,\,
\mathcal{B}_{1}\cap\,B_{1/2}.
$$
Since
\begin{align*}
w_{0}(x)=&\int_{\mathcal{B}_0}D_{y_i}G (x,y)\tilde f_i(y)\,dy\\
=&\frac 2 {k_1+1}\sum_{l=0}^{\infty}(\alpha\beta)^{l}
\Big(\int_{\mathcal{B}_0}D_{y_i}\log|(\Phi_{1}\Phi_{2})^{l}(x)-y|\tilde f_i(y)\,dy\\
&-\beta\int_{\mathcal{B}_0}D_{y_i}\log|(\Phi_{2}\Phi_{1})^{l}\Phi_{2}(x)-y|\tilde f_i(y)\,dy\Big)\\
=&\frac 2 {k_1+1}\sum_{l=0}^{\infty}(\alpha\beta)^{l}
\Big(h_{0}((\Phi_{1}\Phi_{2})^{l}(x))-\beta\,h_{0}((\Phi_{2}\Phi_{1})^{l}
\Phi_{2}(x))\Big),
\end{align*}
we have
$$
|Dw_{0}(x)|\le\frac{C}{k_{1}+1}\cdot\frac{1}
{1-(1-\sqrt\varepsilon)|\alpha\beta|}
\qquad\mbox{in}\,\,\mathcal{B}_{1}\cap\,B_{1/2}.
$$
Similarly, in $\mathcal{B}_{1}\cap\,B_{1/2}$,
\begin{align*}
|Dw_{2}(x)|&\le\frac{Ck_2}{(k_{1}+1)(k_{2}+1)}
\cdot\frac{1}{1-(1-\sqrt\varepsilon)|\alpha\beta|},\\
|Dw_{3}(x)|&\le C+
\frac{C}{k_{1}+1}\cdot\frac{1}
{1-(1-\sqrt\varepsilon)|\alpha\beta|}.
\end{align*}
Therefore, we have
\eqref{eq4.24} in $\mathcal{B}_{1}\cap\,B_{1/2}$.

\underline{Estimates in $\mathcal{B}_{2}\cap\,B_{1/2}$:}
In this case, we have $(\Phi_{2}\Phi_{1})(x)\in\mathcal{B}_{2}$ and $\Phi_1(x)\in \cB_1$.
By exactly the same argument as in the case $x\in\mathcal{B}_{1}\cap\,B_{1/2}$, we get \eqref{eq4.24}.
Hence, estimate \eqref{equ_mainresult} is proved.

{\bf Higher derivative estimates.}
As before, we may assume that $C_m=1$. By Lemma \ref{lem2} and \eqref{eq4.43}, $\tilde{f}_{i}$ are piecewise $C^{m-1,\gamma}$, and
$$
\|\tilde f_i\|_{C^{m-1,\gamma}(B_1\cap \cB_j)}\le Ck_j,\,\,j=0,1,2,\,\,i=1,2,3,
$$
which yields
\begin{equation}
                                    \label{eq4.433}
\|h_j\|_{C^{m,\gamma}(B_{3})}+
\|g_j\|_{C^{m,\gamma}(B_{3})}\le C k_j.
\end{equation}
Recall Fa\`{a} di Bruno's formula for composition
$$
\frac{d^{m}}{dt^{m}}f(g(t))=\sum\frac{m!}{s_{1}!s_{2}!\cdots s_{n}!}f^{(s)}(g(t))\Big(\frac{g'(t)}{1!}\Big)^{s_{1}}
\Big(\frac{g''(t)}{2!}\Big)^{s_{2}}\cdots
\Big(\frac{g^{(n)}(t)}{n!}
\Big)^{s_{n}},
$$
where the summation is over all positive integer solutions of the Diophantine equation
$$
s_{1}+2s_{2}+\cdots+ns_{n}=m
$$
and $s:=s_{1}+s_{2}+\cdots+s_{n}$.
Using \eqref{Phi21m}, we obtain
\begin{align*}
&|D^{m}_{x}h_{1}((\Phi_{2}\Phi_{1})^{l}(x))|\\
&\le\,\sum C|D_{x}(\Phi_{2}\Phi_{1})^{l}(x)|^{s_{1}}
\cdot|D^{2}_{x}(\Phi_{2}\Phi_{1})^{l}(x)|^{s_{2}}
\cdot|D^{n}_{x}(\Phi_{2}\Phi_{1})^{l}(x)|^{s_{n}}\le\frac{Cl^{m-1}}{\lambda_{2}^{2l}},
\end{align*}
where $C>0$ is a constant depending only on $m$.
Therefore, for instance in $\cB_0\cap B_{1/2}$, using \eqref{def_w1} and \eqref{eq4.433},
\begin{align*}
|D^{m}w_{1}(x)|=&\frac{2}{k_{1}+1}
\sum_{l=0}^{\infty}|\alpha\beta|^{l}\left|\Big(D^{m}_{x}h_{1}
((\Phi_{2}\Phi_{1})^{l}(x))-\beta\,D^{m}_{x}h_{1}((\Phi_{2}\Phi_{1})^{l}
\Phi_{2}(x))\Big)\right|\\
\le&\frac{C{k_1}}{k_{1}+1}\Big(1+\sum_{l=1}^{\infty}|\alpha\beta|^{l}
\frac{Cl^{m-1}}{(1+2\sqrt{\varepsilon})^{2l}}\Big)\\
\le&\frac{C{k_1}}{k_{1}+1}\cdot\frac{1}{\big(
1-(1-\sqrt\varepsilon)|\alpha\beta|\big)^m}.
\end{align*}
In the same way, we bound $|D^mw_{0}|$, $|D^m w_{2}|$, and $|D^m w_3|$ in all the three regions, then for $|D^m \tilde{u}|$. Therefore, we obtain the upper bound \eqref{equ_mainresultm} for higher-order derivatives  $|D^{m}\tilde u|$.

{\bf Proof of the claim.}
We consider the growth property of $\tilde u$ as $x\to \infty$.
First we estimate $w_1$. From the definition of $h_1$ in \eqref{eq3.12} and the fact that $\tilde f_i\in C^\alpha(\cB_1)$, we have $|h_1(x)|\le C/(1+|x|)$.
It then follows from \eqref{def_w1} that $\|w_1\|_{L_\infty(\bR^2)}\le C$ for some constant $C$ depending only on $k_1$, $k_2$, and $\varepsilon$.
Similarly, we have
$$
\|w_0\|_{L_\infty(\bR^2)}+\|w_2\|_{L_\infty(\bR^2)}\le C.
$$
On the other hand, from \eqref{eq12.22}, we have $|g_j(x)|\le C\log(|x|+e)$, which gives that $|w_3(x)|\le C\log(|x|+e)$. Therefore,
$$
|\tilde u(x)|\le C\log(|x|+e),
$$
and because of the boundedness of $v$,
\begin{equation}
                            \label{eq3.56}
|\tilde u(x)-v(x)|\le C\log(|x|+e).
\end{equation}
Since $\tilde u-v$ satisfies the homogeneous equation $D_i(aD_i(\tilde u-v))=0$, by the De Giorgi-Nash-Moser estimate, for any $R>0$, we have for some $\alpha_0\in (0,1)$,
$$
[\tilde u-v]_{C^{\alpha_0}(B_{R/2})}\le CR^{-\alpha_0}\|\tilde u-v\|_{L_\infty(B_R)},
$$
which goes to zero as $R\to \infty$ thanks to \eqref{eq3.56}.
The claim is proved.

The proof of Theorem \ref{mainthm1} is completed.
\end{proof}

\subsection{A lower bound for the gradient}\label{subsec_lowerbound}

Assume that $k_1,k_2\ge 1$. Choose $f_1$ to be a nonnegative function supported in $B_{1/10}(-3)$ even in $x_2$ with unit integral, and $f_2\equiv 0$. Define
$$
u(x):=-\int_{B_{1/10}(-3)}D_{y_1}G (x,y)f_1(y)\,dy
$$
and
$$
h(x):=-\int_{B_{1/10}(-3)}D_{y_1}\log|x-y|f_1(y)\,dy.
$$
Then
\begin{align*}
u(x)=&h(x)+\sum_{l=1}^\infty\Big[(\alpha\beta)^l \Big(h((\Phi_1\Phi_2)^l(x))+
h((\Phi_2\Phi_1)^l(x))\Big)\\
&-(\alpha\beta)^{l-1} \Big(\beta h((\Phi_2\Phi_1)^{l-1}\Phi_2(x))+
\alpha h((\Phi_1\Phi_2)^{l-1}\Phi_1(x))\Big)\Big].
\end{align*}
Denote
$$
S:=\{(x_1,0):x_1\in [-(\varepsilon+\varepsilon^2/4),(\varepsilon+\varepsilon^2/4)]\}.
$$
Note that $h$ is also even with respect to $x_2$.
By a simple calculation, we have
\begin{equation}
                                \label{eq3.16}
D_1h(x)\le -c_0<0\quad\text{and}\quad D_2h(x)=0\quad \text{on}\,\,S
\end{equation}
for some constant $c_0>0$. It is easily seen that for any $l\ge 1$,
$$
(\Phi_1\Phi_2)^l(0),\,\,(\Phi_2\Phi_1)^l(0),\,\,
(\Phi_2\Phi_1)^{l-1}\Phi_2(0),\,\, (\Phi_1\Phi_2)^{l-1}\Phi_1(0)\in S
$$
and from the proof of Lemma \ref{lem_Phi_bound} (cf. \eqref{firstderivative} and \eqref{firstderivative12}),
\begin{align*}
&D_1(\Phi_1\Phi_2)^l(0)\ge C/\lambda^{2l},\qquad D_1(\Phi_2\Phi_1)^l(0)\ge C/\lambda^{2l},\\
&D_1\big((\Phi_2\Phi_1)^{l-1}\Phi_2\big)(0)\le -C/\lambda^{2l}, \qquad
D_1\big((\Phi_1\Phi_2)^{l-1}\Phi_1\big)(0)\le -C/\lambda^{2l}
\end{align*}
for some $C>0$ independent of $l(\ge 1)$.
Therefore, by the chain rule,
\begin{align*}
D_1u(0)=&D_1h(0)+\sum_{l=1}^\infty\Big[(\alpha\beta)^l \Big((D_1h)((\Phi_1\Phi_2)^l(0))D_1(\Phi_1\Phi_2)^l(0)\\
&+(D_1h)((\Phi_2\Phi_1)^l(0))D_1(\Phi_2\Phi_1)^l(0)\Big)\\
&-(\alpha\beta)^{l-1} \Big(\beta (D_1h)((\Phi_2\Phi_1)^{l-1}\Phi_2(0))
D_1\big((\Phi_2\Phi_1)^{l-1}\Phi_2\big)(0)\\
&+\alpha D_1h((\Phi_1\Phi_2)^{l-1}\Phi_1(0))
D_1\big((\Phi_1\Phi_2)^{l-1}\Phi_1\big)(0)\Big)\Big].
\end{align*}
which, in view of \eqref{eq3.16}, is less than $-c_0$ multiplied by
$$
1+C\sum_{l=1}^\infty\Big[(\alpha\beta)^l\lambda^{-2l}
+(\alpha\beta)^{l-1} \Big(\beta \lambda^{-2l}+\alpha \lambda^{-2l}\Big)\Big]\ge \frac C {1-(1-\sqrt\varepsilon)\alpha\beta}.
$$
Therefore, \eqref{equ_mainresult_lower} is proved.

\subsection{Proof of Theorem \ref{mainthm2}}
The proof is similar to that of Theorem \ref{mainthm1} with some modifications, which we shall point out. We assume that $C_1=1$.
Take a cutoff function $\eta\in{C}_{0}^{\infty}(\cD_2)$ such that $\eta=1$ on $\cD_1$. Then $v:=u\eta$ satisfies \eqref{equ410}, where
$$
\tilde{f}_{i}=f_{i}\eta+uD_{i}\eta,\quad \tilde f_3=-f_{i}D_{i}\eta+D_{i}uD_{i}\eta.
$$
By the interior estimates for elliptic equations with constant coefficients, instead of \eqref{eq11.56b} we have
\begin{equation}
                            \label{eq11.56d}
\|\tilde f_i\|_{C^\gamma(\cB_j)}\le C\min\{k_j, 1\},\,\,j=0,1,2,\,\,i=1,2,3.
\end{equation}
Now we define $\tilde u$, $w_i,i=0,\ldots,3$ as in \eqref{def_tildeu} as well as $h_j$ and $g_j$ in \eqref{eq3.12} and \eqref{eq12.22}. Using \eqref{eq11.56d}, it holds that
\begin{equation}
                        \label{eq3.28d}
\|h_j\|_{C^{1,\gamma}(B_{3})}+
\|g_j\|_{C^{1,\gamma}(B_{3})}\le C \min\{k_j,1\}.
\end{equation}
As in the proof of Theorem \ref{mainthm1}, in $\cB_0\cap B_{1/2}$, by using \eqref{eq3.28d} we have
\begin{align*}
&|Dw_1(x)|\le \frac{C\min\{k_1,1\}}{k_{1}+1}\cdot
\frac{1}{1-(1-\sqrt\varepsilon)|\alpha\beta|},\\
&|Dw_2(x)|\le \frac{C\min\{k_2,1\}}{k_{2}+1}\cdot
\frac{1}{1-(1-\sqrt\varepsilon)|\alpha\beta|},\\
&|Dw_0(x)|+|Dw_3(x)|\le
\frac{1}{1-(1-\sqrt\varepsilon)|\alpha\beta|},
\end{align*}
which yield \eqref{eq4.24}.
While in $\cB_1\cap B_{1/2}$, we have
\begin{align*}
&|Dw_1(x)|\le \frac {C\min\{k_1,1\}} {k_1}+
\frac{C}{(k_{1}+1)^{2}}\cdot\frac{1}
{1-(1-\sqrt\varepsilon)|\alpha\beta|},\\
&|Dw_2(x)|\le \frac{C\min\{k_2,1\}}{(k_{1}+1)(k_{2}+1)}
\cdot\frac{1}{1-(1-\sqrt\varepsilon)|\alpha\beta|}
,\\
&|Dw_0(x)|+|Dw_3(x)|\le
\frac{C}{k_{1}+1}\cdot\frac{1}
{1-(1-\sqrt\varepsilon)|\alpha\beta|},
\end{align*}
which yield
$$
|D\tilde u(x)|\le
\frac{C}{k_{1}+1}\cdot\frac{1}
{1-(1-\sqrt\varepsilon)|\alpha\beta|}.
$$
Similarly, in $\cB_2\cap B_{1/2}$,
$$
|D\tilde u(x)|\le
\frac{C}{k_{2}+1}\cdot\frac{1}
{1-(1-\sqrt\varepsilon)|\alpha\beta|}.
$$
The theorem is proved.

\section{Proof of Theorem \ref{mainthm3}}\label{sec_thm3}

For general $r_{1}$ and $r_{2}$, the inversion maps with respect to $\partial B_1$ and $\partial B_2$ in the complex variable are given by
\begin{equation*}
\Phi_1(z):=\frac{r_{1}^{2}}{\bar z-(r_{1}+\varepsilon/2)}+r_{1}+\varepsilon/2,\qquad
\Phi_2(z):=\frac{r_{2}^{2}}{\bar z+(r_{2}+\varepsilon/2)}-(r_{2}+\varepsilon/2).
\end{equation*}
Then
\begin{align*}
&(\Phi_{2}\Phi_{1})(z):=\Phi_{2}\circ\Phi_{1}(z)\\
&=\frac{r_{2}^{2}(z-(r_{1}+\varepsilon/2))}{(r_{1}+r_{2}+\varepsilon) z-(r_{1}+\varepsilon/2)(r_{1}+r_{2}+\varepsilon)+r_{1}^{2}}-(r_{2}+\varepsilon/2).
\end{align*}
Using a translation and dilation of coordinates
$$
(r_{1}+r_{2}+\varepsilon) z-(r_{1}+\varepsilon/2)(r_{1}+r_{2}+\varepsilon)+r_{1}^{2}\to z
$$
and
$$(r_{1}+r_{2}+\varepsilon) (\Phi_{2}\Phi_{1})(z)-(r_{1}+\varepsilon/2)(r_{1}+r_{2}+\varepsilon)+r_{1}^{2}\to (\Phi_{2}\Phi_{1})(z),$$ we obtain
$$
(\Phi_{2}\Phi_{1})(z):=-\frac{r_{1}^{2}r_{2}^{2}}{z}-\Big(2r_{1}r_{2}+2(r_{1}+r_{2})\varepsilon+\varepsilon^{2}\Big).
$$
The two fixed points of the map $(\Phi_{2}\Phi_{1})$ are given by
\begin{align*}
\lambda_{1}&=-(r_{1}r_{2}+(r_{1}+r_{2})\varepsilon+\varepsilon^{2}/2)\\
&\quad +\sqrt{r_{1}r_{2}\Big(2(r_{1}+r_{2})\varepsilon+\varepsilon^{2}\Big)
+\Big((r_{1}+r_{2})\epsilon+\varepsilon^{2}/2\Big)^{2}}\\
&\sim r_{1}r_{2}\Big(-1+\sqrt{2(1/r_{1}+1/r_{2})\varepsilon}\Big)
\quad\text{in}\,\,\mathcal{B}_{1}
\end{align*}
and
$$
\lambda_{2}\sim r_{1}r_{2}\Big(-1-\sqrt{2(1/r_{1}+1/r_{2})\varepsilon}\Big)
\quad\text{in}\,\,\mathcal{B}_{2}.
$$

Since for $x\in\mathcal{B}_{0}\cap\,B_{1/2}$, $(\Phi_{2}\Phi_{1})^{l}(x)$ and $(\Phi_{2}\Phi_{1})^{l-1}\Phi_{2}(x)$, $l\geq1$, are all in $\mathcal{B}_{2}$, we here choose $\lambda_{2}$. By a similar calculation as before,
$$
\frac 1 {(\Phi_{2}\Phi_{1})(z)-\lambda_{2}}-\frac{\lambda_{2}}{r_{1}^{2}r_{2}^{2}-\lambda_{2}^2}=
\frac{\lambda_{2}^2}{r_{1}^{2}r_{2}^{2}}\Big(\frac{1}{z-\lambda_{2}}-\frac{\lambda_{2}}{r_{1}^{2}r_{2}^{2}-\lambda_{2}^2}\Big).
$$
By iteration, we have for any $z\neq \lambda_{2}$ and $k\ge 0$,
\begin{align*}
(\Phi_{2}\Phi_{1})^{l}(z)&=\lambda_{2}
+\big(\lambda_{2}-r_{1}^{2}r_{2}^{2}/\lambda_2\big)
\big(r_{1}r_{2}/\lambda_{2}\big)^{2l}\cdot \frac{1}{1-\big(r_{1}r_{2}/\lambda_{2}\big)^{2l}}\\
&\quad\cdot\Big(1+\frac{r_{1}^{2}r_{2}^{2}\lambda_2^{-1}-\lambda_{2}}
{z\Big(1-\big(r_{1}r_{2}/\lambda_{2}\big)^{2l}\Big)
-\Big(r_{1}^{2}r_{2}^{2}\lambda_2^{-1}-\lambda_{2}
\big(r_{1}r_{2}/\lambda_{2}\big)^{2l}\Big)}\Big).
\end{align*}
Therefore,
\begin{align*}
&D(\Phi_{2}\Phi_{1})^{l}(z)
=(\lambda_{2}-r_{1}^{2}r_{2}^{2}/\lambda_{2})^2
\big(r_{1}r_{2}/\lambda_{2}
\big)^{2l}\\
&\qquad\cdot\Big(\big(z-r_{1}^{2}r_{2}^{2}\lambda_2^{-1}\big)\Big(1-
\big(r_1r_2/\lambda_{2}\big)^{2l}\Big)
+(\lambda_{2}-r_{1}^{2}r_{2}^{2}\lambda_2^{-1})\big(r_1r_2
/\lambda_{2}\big)^{2l}\Big)^{-2}.
\end{align*}
Since
$$
\lambda_{2}-r_{1}^{2}r_{2}^{2}\lambda_2^{-1}\sim -r_{1}r_{2}\sqrt{2(1/r_{1}+1/r_{2})\varepsilon}
$$
and for $z\in \mathcal{B}_{2}$ (in the new coordinates),
\begin{align*}
&\mathrm{Re}~ (z-r_{1}^{2}r_{2}^{2}\lambda_2^{-1})\\
&\le -(r_{1}+r_{2}+\varepsilon)\varepsilon/2-(r_{1}+\varepsilon/2)(r_{1}+r_{2}+\varepsilon)+r_{1}^{2}-r_{1}^{2}r_{2}^{2}\lambda_2^{-1}\\
&=r_{1}r_{2}\Big(\frac{1}{1+\sqrt{2(1/r_{2}+1/r_{2})\varepsilon}}-1\Big)+O(\varepsilon) \\
&\lesssim -r_{1}r_{2}\sqrt{(1/r_{1}+1/r_{2})\varepsilon},
\end{align*}
we have
\begin{align*}
&\Big|\big(z-r_{1}^{2}r_{2}^{2}\lambda_2^{-1}\big)\Big(1-
\big(r_1r_2/\lambda_{2}\big)^{2l}\Big)
+(\lambda_{2}-r_{1}^{2}r_{2}^{2}\lambda_2^{-1})
\big(r_1r_2/\lambda_{2}\big)^{2l}\Big|\\
&\gtrsim r_1r_2\sqrt{(1/r_1+1/r_2)\varepsilon}.
\end{align*}
Thus, for $l\geq1$,
$$
|D(\Phi_{2}\Phi_{1})^{l}(z)|\le \frac{C}{\big(1+\sqrt{2(1/r_{1}
+1/r_{2})\varepsilon}\big)^{2l}}.
$$
Similarly, we can bound $D(\Phi_1\Phi_2)$ as well as the higher derivatives of $\Phi_1\Phi_2$ and $\Phi_2\Phi_1$.
Thus we obtain a generalization of Lemma \ref{lem_Phi_bound}, which shows the dependence of $r_{1}$ and $r_{2}$.

\begin{lemma}\label{lem_Phi_boundr1r2}
For any integers $l\ge 0$ and $m\ge 1$,
\begin{equation*}
|D^m(\Phi_{2}\Phi_{1})^{l}(z)|\le \frac{Cl^{m-1}}{\big(1+\sqrt{2(1/r_{1}+1/r_{2})\varepsilon}\big)^{2l}}\quad
\text{in}\,\,\mathcal{B}_{2},
\end{equation*}
and
\begin{equation*}
|D^m(\Phi_{1}\Phi_{2})^{l}(z)|\le \frac{Cl^{m-1}}{\big(1+\sqrt{2(1/r_{1}+1/r_{2})\varepsilon}\big)^{2l}}\quad
\text{in}\,\,\mathcal{B}_{1},
\end{equation*}
where $C$ depends only on $m$.
\end{lemma}

Proposition \ref{prop1} still holds with obvious modifications. Using Lemma \ref{lem_Phi_boundr1r2} instead of Lemma \ref{lem_Phi_bound}, by the same
procedure as in the proof of Theorem \ref{mainthm2}, Theorem \ref{mainthm3} is proved.

\medskip

\noindent {\bf Funding:} Hongjie Dong was partially supported by the NSF under agreement DMS-1056737. Haigang Li was partially supported by  NSFC (11571042) (11371060), Fok Ying Tung Education Foundation (151003), and the Fundamental Research Funds for the Central Universities.

\noindent {\bf Conflict of Interest:} The authors declare that they have no conflict of interest.

\bibliographystyle{plain}

\begin{thebibliography}{10}

\bibitem{adn} S. Agmon; A. Douglis; L. Nirenberg, Estimates near the boundary for solutions of elliptic partial differential equations satisfying general boundary conditions. II. Comm. Pure Appl. Math. 17 (1964) 35--92.

\bibitem{abtv} H. Ammari; E. Bonnetier; F. Triki; M. Vogelius, Elliptic estimates in composite media with smooth inclusions: an integral equation approach. Ann. Sci. ƒc. Norm. Sup\'er. (4) 48 (2015), no. 2, 453--495.

\bibitem{ackly} H. Ammari; G. Ciraolo; H. Kang; H. Lee; K. Yun, Spectral analysis of the Neumann-Poincar\'{e} operator and characterization of the stress concentration in anti-plane elasticity. Arch. Ration. Mech. Anal.  208  (2013),  275--304.

\bibitem{akl} H. Ammari; H. Kang; M. Lim, Gradient estimates for solutions to the conductivity problem. Math. Ann. 332 (2005), 277--286.

\bibitem{aklll}  H. Ammari; H. Kang; H. Lee; J. Lee; M. Lim, Optimal estimates for the electrical
field in two dimensions. J. Math. Pures Appl. 88 (2007), 307--324.

\bibitem{akllz} H. Ammari; H. Kang; H. Lee; M. Lim; H. Zribi, Decomposition theorems and fine estimates for electrical
fields in the presence of closely located circular inclusions. J.
Differential Equations 247 (2009), 2897--2912.

\bibitem{ba} I. Babu\v{s}ka; B. Andersson; P. Smith; K. Levin, Damage analysis of fiber composites. I. Statistical analysis on fiber scale.
Comput. Methods Appl. Mech. Engrg. 172 (1999), 27--77.

\bibitem{bly1} E.S. Bao; Y.Y. Li; B. Yin, Gradient estimates for the perfect conductivity problem. Arch. Ration. Mech. Anal. 193 (2009), 195--226.

\bibitem{bly2} E.S. Bao; Y.Y. Li; B. Yin, Gradient estimates for the perfect and insulated conductivity problems with multiple inclusions. Comm. Partial Differ. Equ. 35 (2010), no. 11, 1982--2006.

\bibitem{bll} J.G. Bao; H.G. Li; Y.Y. Li, Gradient estimates for solutions of the Lam\'{e} system with partially infinite coefficients. Arch. Ration. Mech. Anal.  215  (2015),  no. 1, 307--351.

\bibitem{bt0}  E. Bonnetier and F. Triki,
Pointwise bounds on the gradient and the spectrum of the Neumann-Poincar\'{e}
 operator: the case of 2 discs,  Multi-scale and high-contrast PDE: from modeling, to mathematical analysis, to inversion,  81--91, Contemp. Math., 577, Amer. Math. Soc., Providence, RI, 2012.

\bibitem{bt} E. Bonnetier; F. Triki, On the spectrum of the Poincar\'{e} variational problem for two close-to-touching inclusions in $2D$. Arch. Ration. Mech. Anal. 209 (2013), no. 2, 541--567.

\bibitem{bv} E. Bonnetier; M. Vogelius, An elliptic regularity result for a
composite medium with ``touching'' fibers of circular cross-section.
SIAM J. Math. Anal. 31  (2000) 651--677.

\bibitem{BC84} B. Budiansky; G. F. Carrier, High shear stresses in stiff fiber composites. J. Appl. Mech. 51 (1984), 733--735.

\bibitem{d} H. Dong, Gradient estimates for parabolic and elliptic systems from linear laminates. Arch. Ration. Mech. Anal. 205 (2012), no. 1, 119--149.


\bibitem{dz} H. Dong; H. Zhang, On an elliptic equation arising from composite materials. to appear in Arch. Ration. Mech. Anal., DOI 10.1007/s00205-016-0996-9.

\bibitem{kly}  H. Kang; M. Lim; K. Yun, Asymptotics and computation of the solution to the conductivity equation in the presence of adjacent inclusions with extreme conductivities. J. Math. Pures Appl. (9) 99 (2013), 234--249.

\bibitem{kly2}  H. Kang; M. Lim; K. Yun, Characterization of the electric field concentration between two adjacent spherical perfect conductors. SIAM J. Appl. Math.  74  (2014),  no. 1, 125--146.

\bibitem{llby} H.G. Li; Y.Y. Li; E.S. Bao; B. Yin, Derivative estimates of solutions of elliptic systems in narrow regions. Quart. Appl. Math.  72  (2014),  no. 3, 589--596.

\bibitem{ln} Y.Y. Li; L. Nirenberg, Estimates for elliptic systems from composite material. Comm. Pure Appl. Math. 56, (2003) 892--925.

\bibitem{lv} Y.Y. Li; M. Vogelius, Gradient estimates for solutions to divergence form elliptic equations with discontinuous coefficients.
Arch. Rational Mech. Anal. 135 (2000), 91--151.

\bibitem{ly} M. Lim; K. Yun, Blow-up of electric fields between closely spaced spherical perfect conductors, Comm. Partial Differential Equations, 34 (2009), pp. 1287-1315.

\bibitem{Ma96} X. Markenscoff,
Stress amplification in vanishing small geometries. Computational Mechanics, 19 (1996), 77--83.

\bibitem{y1} K. Yun, Optimal bound on high stresses occurring between stiff fibers with arbitrary shaped cross-sections. J. Math. Anal. Appl. 350 (2009), 306--312.

\bibitem{y2} K. Yun, Estimates for electric fields blown up between closely adjacent conductors with arbitrary shape. SIAM J. Appl. Math. 67 (2007),  714--730.

\end{thebibliography}

\def\cprime{$'$}

\end{document}